\documentclass[oneside,a4paper,11pt,reqno]{amsart}
\usepackage{amsfonts}
\usepackage{amssymb}
\usepackage{amsthm}
\usepackage{amsmath}
\usepackage{a4wide}
\usepackage{mathrsfs}
\usepackage{epsfig}
\usepackage{palatino}
\usepackage{tikz,fp,ifthen}
\usepackage{float}
\usepackage{graphicx}
\usepackage{enumitem}
\usepackage{esint}
\usepackage{accents}
\usepackage[colorinlistoftodos]{todonotes}
\usepackage{mathtools}
\usepackage{cite}
\mathtoolsset{showonlyrefs}

\usepackage{color}
\definecolor{citegreen}{rgb}{0,0.6,0}
\definecolor{refred}{rgb}{0.8,0,0}
\usepackage[colorlinks,urlcolor=blue,citecolor=citegreen,linkcolor=refred]{hyperref}
\usepackage{appendix}

\theoremstyle{plain}
\newtheorem{thm}{Theorem}[section]
\newtheorem{lemma}[thm]{Lemma}
\newtheorem{proposition}[thm]{Proposition}
\newtheorem{cor}[thm]{Corollary}
\newtheorem{theorem}[thm]{Theorem}
\newtheorem*{teorema}{Theorem} 

\newtheorem{ackn}{Acknowledgement\!\!}

\theoremstyle{remark}

\newtheorem{remark}[thm]{Remark}

\theoremstyle{definition}
\newtheorem{definition}[thm]{Definition}

\numberwithin{equation}{section}

\newcommand{\R}{\mathbb{R}}
\newcommand{\N}{\mathbb{N}}
\def\Z{\mathbb Z}
\newcommand{\SSS}{\mathbb S}

\newcommand{\RRR}{{\mathrm R}}
\newcommand{\Ric}{{\mathrm {Ric}}}

\newcommand{\pa}{\partial}

\newcommand{\D}{{\rm D}}

\newcommand{\HHH}{\mathrm{H}} 

\newcommand{\na}{\nabla}

\newcommand{\medint}{-\kern -,375cm\int}
\newcommand{\medintinrigo}{-\kern -,315cm\int}

\newcommand{\beq}{\begin{equation}}
\newcommand{\eeq}{\end{equation}}

\def\ringg#1{\accentset{\circ}{#1}}

\begin{document}
\title[]{ADM mass, area and capacity in asymptotically flat $3$--manifolds with nonnegative scalar curvature}

\author[F.~Oronzio]{Francesca Oronzio}
\address{F.~Oronzio, Universit\`a degli Studi di Napoli Federico II, Italy}
\email{francesca.oronzio@unina.it}

\begin{abstract}
We show an improvement of Bray sharp mass--capacity
inequality and Bray--Miao sharp upper bound of the capacity of the boundary in terms of its area, for three--dimensional, complete, one--ended asymptotically flat manifolds
with compact, connected boundary and with nonnegative scalar
curvature, under appropriate assumptions on the topology 
and on the mean curvature of the boundary. Our arguments relies on two
monotonicity formulas holding along level sets of a suitable harmonic
potential, associated to the boundary of the manifold. This work is an expansion of the results contained in the PhD thesis of the author. 
\end{abstract}

\maketitle

\tableofcontents

\section{Introduction}

Monotonicity formulas play an important role in geometric analysis, 
some well--known examples are given by the monotonicity formula for
minimal submanifolds and the classical Bishop--Gromov volume
comparison theorem in the context of comparison geometry, while
Huisken monotonicity formula for the mean curvature
flow~\cite{Hui_1990}, Perelman entropy formula for the Ricci
flow~\cite{Per}, or Geroch monotonicity of the Hawking mass along
the inverse mean curvature flow~\cite{ger73,HI} are extremely 
relevant for geometric flows.
Analogously to the fact that Perelman monotonicity is closely
related to the sharp gradient estimate for the heat kernel of
Li and Yau, in a series of works~\cite{Colding_Acta, Col_Min_2,Col_Min_3}, 
Colding and Minicozzi obtained some monotonicity
formulas along the level sets of the minimal positive Green function in nonparabolic Riemannian manifolds
with nonnegative Ricci curvature, as a consequence of a new sharp
gradient estimate for such function and used
them to prove the uniqueness of tangent cones for Einstein
manifolds~\cite{Col_Min_4}. Afterwards, the same monotone quantities were used to obtain
new Willmore--type geometrical inequalities in~\cite{Ago_Fog_Maz_1}
(see also~\cite{Ago_Maz_CV} for the Euclidean setting) and 
generalized to prove an optimal version of the Minkowski inequality
in~\cite{BFM} (see also~\cite{Ago_Fog_Maz_2} for the Euclidean
setting) for nonparabolic Riemannian manifolds with nonnegative Ricci
curvature. 
With the same idea some new monotonicity formulas (even with
non--harmonic functions) were then found to study static and sub--static
metrics in general relativity~\cite{Virginia1, Bor_Chr_Maz, Bor_Maz_1,Bor_Maz_2, AMO}.
Recently, this level set approach with harmonic functions applied to certain nonparabolic Riemannian $3$--manifolds
with nonnegative scalar curvature have produced several results, we mention 
the sharp comparisons about the rate of decay of an energy--like quantity and the area of the
level sets of the minimal positive Green function, obtained by Munteanu and Wang
in~\cite{MuntWang} which, as an application, allowed Chodosh and
Li~\cite{Chodosh_Li} to prove a conjecture of Schoen on the
stable minimal hypersurfaces in the Euclidean space $\R^4$ and some 
new proofs of the well--known positive mass theorem~\cite{Ago_Maz_Oro_2,bray3}. 
Then, considering linearly growing harmonic functions (similarly to~\cite{bray3}, which was influenced by a pioneering work of Stern~\cite{Stern}), 
some asymptotically flat versions of the {\em spacetime}
positive mass theorem have been proven in~\cite{Hir_Kaz_Khu, Bra_Hir_Kaz_Khu_Zha_2021}, while the
monotonicity formula in~\cite{Ago_Maz_Oro_2} has been extended
in~\cite{AMMO} to the nonlinear potential theoretic setting, replacing
the harmonic functions by $p$--harmonic functions, 
to obtain a simpler proof of the Riemannian Penrose inequality
with a single black hole.

With an argument similar to the one~\cite{Ago_Maz_Oro_2}, we show here a sharp inequality involving the ratio between the ADM mass and the capacity of the boundary and a sharp upper bound on this latter in terms of the area of the boundary, for asymptotically flat Riemannian $3$--manifolds with a single end, with a connected, compact boundary and
nonnegative scalar curvature, under appropriate assumptions on the topology and on the mean curvature of the boundary.  
One of the reasons for the interest in these mass--capacity inequalities is to apply them to obtain generalizations of
Bunting and Masood--ul--Alam rigidity theorem~\cite[Theorem~2]{ButMas}, as in~\cite{Miao2005} and~\cite[Section 5]{AMO} (based on results in~\cite{bray1,hirsch}).\\
Our first inequality is an extension of the cases of validity and a refinement, if certain topological assumptions are satisfied, of the one obtained by Bray by means of the positive mass theorem in~\cite{bray1}, while our upper bound on the capacity of the boundary is an improvement, if an appropriate assumption on the mean curvature of the boundary is fulfilled, of the one proved by Bray and Miao in~\cite{bray2} through the technique of the weak inverse mean curvature flow, developed by Huisken and Ilmanen in~\cite{HI}.
Being our inequalities consequence of two ``elementary'' monotonicity formulas holding along the level sets of an appropriate harmonic function, the proofs are more direct and self--contained.\\
We mention that during the preparation of this paper, Miao~\cite{Miao2022} obtained similar inequalities with the same approach.

\smallskip

In order to state precisely our main results, we recall the definition of one--ended asymptotically flat manifold and of ADM mass 
(after the names of R.~Arnowitt, S.~Deser and C.~W.~Misner, who introduced it in~\cite{AdM0}).

\begin{definition}\label{defAFman}
A $3$--dimensional Riemannian manifold $(M,g)$ (with or without boundary) is said to be {\em one--ended asymptotically flat} if there exists a closed and bounded subset $K$ and a diffeomorphism $\Phi: M\setminus K\to\R^{3}\setminus\overline{B}_{r}(0)$ such that in the (coordinate) chart $(M\setminus K, \Phi=(x^{i}))$, called {\em asymptotically flat (coordinate) chart}, setting $g=g_{ij}\,dx^{i}\otimes dx^{j}$, there holds
\begin{equation}\label{eq4}
g_{ij}=\delta_{ij} + O_{2}(\vert x \vert ^{-\tau})\,,
\end{equation}
for some constant $\tau>1/2$ (the {\em order of decay} of $g$ in the asymptotically flat chart $(M\setminus K, (x^{i}))$, briefly {\em the order}), where $\delta$ is the Kronecker delta function.
\end{definition}

Here and in the rest of this work, the {\em(exterior spatial) Schwarzschild manifold of mass $m>0$} is the $3$--dimensional Riemannian manifold with boundary given by the couple
\begin{equation}\label{feq63bisbis}
\Big(\R^{3}\setminus B_{\frac{ m}{2}}(0),\,\Big(1+\frac{m}{\,2\vert  x\vert}\Big)^{4}g_{{\mathrm{eucl}}}\Big)\,,
\end{equation}
which will sometimes be denoted by $(M_{\mathrm{Sch}(m)},\, g_{\mathrm{Sch}(m)})$. This space is easily seen to be a (model) nontrivial example of a one--ended asymptotically flat, $3$--dimensional Riemannian manifold with minimal boundary.

\begin{definition}\label{defADMmass}
Let $(M,g)$ be a one--ended asymptotically flat manifold having integrable or nonnegative scalar curvature and let $\big(E,(x^{1},x^{2},x^{3})\big)$ be an asymptotically flat chart.
The limit
\begin{equation}\label{formADMmass}
m_{\mathrm{ADM}}=\lim_{r\to +\infty}\frac{1}{16\pi}\!\!\!\int\limits_{\{\vert  x\vert\,=\,r\}}\!\!\!(\partial_{j}g_{ij}-\partial_{i}g_{jj})\frac{x^{i}}{\vert  x \vert}\,d\sigma_{\mathrm{eucl}}\,,
\end{equation}
where $g=g_{ij}dx^{i}\otimes dx^{j}$, exists (possibly equal to $+\infty$) and it is independent of the asymptotically flat chart (proved first by Bartnik~\cite{Bartnik} and then independently by Chru\'sciel~\cite{Chrusciel}).\\
This geometric invariant is called {\em ADM mass} of $(M,g)$. 
\end{definition}

This paper presents the current state of the work (in progress) in generalizing and extending the results contained in the PhD thesis of the author, under the supervision of Virginia Agostiniani, Carlo Mantegazza and Lorenzo Mazzieri. Precisely, we will show the following conclusions.

\begin{teorema}[Theorem~\ref{teo2 lower bound differenza tra il rapporto massa cap e una normalizzazione di Willmore energybis}, Theorem~\ref{lower bound differenza tra il rapporto massa cap e una normalizzazione di Willmore energybis} and Corollary~\ref{genBray}]
Let $(M,g)$ be a $3$--dimensional, complete, one--ended asymptotically flat manifold with compact, connected boundary and with nonnegative scalar curvature.
We consider the solution $u\in C^{\infty}(M)$ of the Dirichlet problem
\begin{equation}
\begin{cases}\label{f1prelc4}
\Delta u=0 \ &\mathrm{in} \ M\,\\
u=0 &\mathrm{on} \ \partial M\,\\
u\to 1&\mathrm{at} \ \infty\,
\end{cases}
\end{equation}
and the {\em boundary capacity} of $\partial M$ in $(M,g)$, defined as
\begin{equation}\label{boundcapintermofu}
\mathcal{C}=\frac{1}{4\pi}\int\limits_{\partial M}\!\vert \na u \vert d\sigma=\frac{1}{4\pi}\int\limits_{M}\!\vert \na u \vert^{2}d\mu\,.
\end{equation}
Assume that $H_{2}(M,\partial M;\Z)=0$, then,
\begin{equation}\label{introeq1}
\frac{m_{\mathrm{ADM}}}{\mathcal{C}}\,\geq\,\frac{5}{4}\,+\,\frac{1}{64\pi}\!\int\limits_{\partial M}\!\mathrm{H}^{2}\,d\sigma\,-\,\frac{1}{4\pi}\!\int\limits_{\partial M}\!\bigg(\vert \nabla u \vert\,+\,\frac{\mathrm{H}}{4}\bigg)^{2}\, d\sigma\,,
\end{equation}
with equality if and only if $(M,g)$ is isometric to a (exterior spatial) Schwarzschild manifold of mass $m>0$.
Moreover, if there exists $\alpha\in\big(\!-(2\mathcal{C})^{-1},\,(2\mathcal{C})^{-1}\big]$ such that 
$ \HHH\leq \alpha\big(1-4\mathcal{C}\vert  \nabla u \vert\,\big)$ on $\partial M$, the term on the right hand side of the inequality is greater or equal than $1$, hence
\begin{equation}
m_{\mathrm{ADM}}\,\geq\,\mathcal{C}\,,
\end{equation}
with equality if and only if $(M,g)$ is isometric to a (exterior spatial) Schwarzschild manifold of mass $m>0$.
\end{teorema}

\begin{teorema}[Theorem~\ref{teocapacityandareaboundarybis}]
Let $(M,g)$ be a $3$--dimensional, complete, one--ended asymptotically flat manifold with compact, connected boundary and with nonnegative scalar curvature.
Let $u\in C^{\infty}(M)$ be the solution of Dirichlet problem~\eqref{f1prelc4} and let $\mathcal{C}>0$ be the boundary capacity of $\partial M$ in $(M,g)$, given by formula~\eqref{boundcapintermofu}.
Assume that $H_2(M,\partial M;\Z) = 0$ and that there exists $\alpha\in\big(\!-(2\mathcal{C})^{-1},\,(2\mathcal{C})^{-1}\big]$ such that 
$ \HHH\leq \alpha\big(1-4\mathcal{C}\vert  \nabla u \vert\,\big)$ on $\partial M$, then, 
\begin{equation}
\sqrt{\frac{\mathrm{Area}(\partial M)}{16\pi}}\geq\mathcal{C}\,,
\end{equation}
with equality if and only if $(M,g)$ is isometric to a (exterior spatial) Schwarzschild manifold of mass $m>0$.
\end{teorema}

\bigskip

\begin{ackn}
We are really grateful to Riccardo Benedetti and Stefano Borghini for some precious discussions.
We also want to deeply thank Virginia Agostiniani, Carlo Mantegazza and Lorenzo Mazzieri for their interest and support in the present work and for the useful suggestions and comments about this manuscript.
The author is member of the Gruppo Nazionale per l'Analisi
Matematica, la Probabilit\`a e le loro Applicazioni (GNAMPA), which is part of the Istituto
Nazionale di Alta Matematica (INdAM).
\end{ackn}

\section{Preliminaries}\label{settmasscapacity}

For the convenience of the readers, in this small section we collect some basic facts about the solution $u\in C^{\infty}(M)$ of Dirichlet problem~\eqref{f1prelc4} in $(M,g)$, defined as above,
and its level sets.

By the maximum principle,
$$
\mathrm{Int}(M)=M\setminus \partial M=\{0<u<1\}\,.
$$
It follows then $\partial M=\{u=0\}$ and from the Hopf lemma that zero is a regular value of $u$.  

The last condition in system~\eqref{f1prelc4} implies that $u:M\to[0,1)$ is proper.
Then, some consequences are the compactness of each level set of $u$, which leads to the finiteness of their $2$--dimensional Hausdorff measure of $(M,g)$ (see~\cite[Theorem~1.7]{Hardt1}), and the fact that the set of the regular values of $u$ is an open set of $[0,1)$.
Hence, the set of the critical values of $u$ is a closed set of $[0,1)$ and has zero Lebesgue measure, by Sard theorem.

It is known that in a generic asymptotically flat chart $(x^{1},x^{2},x^{3})$ of order $\tau$, $\tau>1/2$, one has
\begin{equation}\label{hexpu}
u=1-\frac{\mathcal{C}}{\vert x\vert}+O_{2}(\vert x\vert^{-1-\beta})
\end{equation}
for some $1/2<\beta<\min\{\tau,1\}$, as $1-u$ is {\em the boundary capacity potential},~\cite[Lemma A.2]{MMT}. Here, $\mathcal{C}$ is the boundary capacity of $\partial M$ in $(M,g)$, given by formula~\eqref{boundcapintermofu} and which also satisfies
\begin{equation}\label{feq44}
\int\limits_{\{u\,=\,s\}}\!\!\!\vert \na u\vert\, d\sigma=4\pi\mathcal{C}\,
\end{equation}
for a.e. $s\in[0,1)$, in particular, any $s$ regular value for $u$, by the divergence theorem and Sard theorem.
A consequence of formula~\eqref{hexpu} is the compactness of $\mathrm{Crit}(u)$, therefore, $\mathrm{Crit}(u)$ has finite $1$--dimensional Hausdorff measure (see~\cite[Theorem~1.1]{Hardt2}).
Also, it follows from formula~\eqref{hexpu} that the level sets $\{u = t\}$ are diffeomorphic to the sphere $\SSS^{2}$ for every $t\in [0, 1)$ sufficiently close to $1$, once it is proved that they are connected. For these last two facts we refer to~\cite[Remark~2.1]{AMO}.

Finally, since the boundary $\partial M$ of $M$ is connected, if $M$ has a simple topology, namely $H_{2}(M,\partial M;\Z)=0$, then all regular level sets of $u$ are closed, connected surfaces.
A proof of this result can be found in~\cite[Subsection~1.3]{AMMO}, but, in Appendix~\ref{AppA}, one is presented based on the fact that the first Betti number $b_{1}(M)$ is zero. Indeed, for this type of manifolds the two conditions $H_{2}(M,\partial M;\Z)=0$ and $b_{1}(M)=0$ are equivalent and they also imply that $M$ is orientable and $\partial M$ is a $2$--sphere (up to diffeomorphism).

\section{First monotonicity formula}\label{SecMonforFandRigstat1}

In this section we are going to prove our first monotonicity formula. It is a natural version with boundary and for the comparison with the (exterior spatial) Schwarzschild manifold of mass $\mathcal{C}$ (see formula~\eqref{boundcapintermofu}) of the one shown in~\cite{Ago_Maz_Oro_2}.
The main difficulty amounts to ensure that the monotonicity survives the critical
values of $u$ (solution of Dirichlet problem~\eqref{f1prelc4}), that, as already recalled in Section~\ref{settmasscapacity}, form a set of zero Lebesgue measure. 
The approach followed to overcome this difficulty is the same of~\cite{Ago_Maz_Oro_2}, namely via a sequence of appropriate cut--off functions.
This monotonicity formula and the analysis of when it is constant play a key role in obtaining inequality~\eqref{introeq1} with the equality case and are the subject of the following two propositions.

\begin{proposition}\label{propomonGcap}
Let $(M,g)$ be a $3$--dimensional, complete, one--ended asymptotically flat manifold with compact, connected boundary and with nonnegative scalar curvature.
Let $u\in C^{\infty}(M)$ be the solution of Dirichlet problem~\eqref{f1prelc4} and let $\mathcal{C}>0$ be the boundary capacity of $\partial M$ in $(M,g)$ given by formula~\eqref{boundcapintermofu}.
Consider the function $F:[\mathcal{C}/2, + \infty)\to \R$ defined as 
\begin{equation}\label{feq40}
F(t)=4\pi t\,+ \,\frac{t^{3}}{\mathcal{C}^{2}} \,\left(1+\frac{\mathcal{C}}{2t}\right)^{\!3} \!\left(1-\frac{3\mathcal{C}}{2t}\right)\int\limits_{\Sigma_{t}} \vert \nabla u \vert^{2}\, d\sigma\,-\,\frac{t^{2}}{\mathcal{C}} \,\left(1+\frac{\mathcal{C}}{2t}\right)^{\!2}\!\int\limits_{\Sigma_{t}}\vert \nabla u \vert\,\mathrm{H}\, d\sigma\,,
\end{equation}
where $\Sigma_{t}$ is the level set of $u$, given by
\begin{equation}\label{Sigmat}
\Sigma_{t}:=\Big\{u=\Big(1-\frac{\mathcal{C}}{2t}\Big)/\Big(1+\frac{\mathcal{C}}{2t}\Big)\Big\}\,,
\end{equation}
$\mathrm{H}$ is the mean curvature of $\Sigma_{t}\setminus\mathrm{Crit}(u)$ with respect to the $\infty$--pointing unit normal vector field $\nu={\nabla u}/{\vert \nabla u \,\vert}$ and $\sigma$ is the $2$--dimensional Hausdorff measure of $(M,g)$.
Then, if all regular level sets of $u$ are connected, $F$ is nondecreasing on the set $\mathcal{T}$, defined by
\begin{equation}\label{defTcap}
\mathcal{T}:=\Big\{t\in[\mathcal{C}/2,+\infty)\,:\, \Big(1-\frac{\mathcal{C}}{2t}\Big)/\Big(1+\frac{\mathcal{C}}{2t}\Big) \text{ is a regular value of $u$}\Big\}\,.
\end{equation}
\end{proposition}

Notice that the function $F$ is well defined, indeed, all level sets have finite $\sigma$--measure, as already observed in Section~\ref{settmasscapacity}, and the integrand functions are $\sigma$--a.e. bounded on each level set of $u$.
Indeed, one has 
\begin{equation}
\label{eq:hdubound1}
\big\vert \vert \nabla u \vert \mathrm{H}\big\vert \leq \vert  \nabla du (\nu,\nu) \vert\leq \vert  \nabla du  \vert \, \quad 
\end{equation}
wherever $\vert \nabla u\vert \neq 0$, since 
\begin{equation}\label{H1}
\mathrm{H}=- \frac{\nabla du (\nabla u,\nabla u)}{\vert \nabla u\vert ^{3}}=-\frac{g(\nabla \vert \nabla u\vert , \nabla u)}{ \vert \nabla u\vert^{2}}\,,
\end{equation}
as $u$ is harmonic.

Observe also that the set $\mathcal{T}$ differs from $[\mathcal{C}/2,+\infty)$ only for a negligible set and is a disjoint countable union of open intervals and of only one interval of type $[a, b)$, with $a=\mathcal{C}/2$, since the set of the regular values of $u$ is an open set of $[0,1)$ and $\partial M=\{u=0\}$ is a regular level set of $u$, as explained in Section~\ref{settmasscapacity}. 
In $\mathcal{T}$ the function $F$ is continuously differentiable, with first derivative given by
\begin{equation}\label{feq54}
F'(t)=4\pi- \int\limits_{\Sigma_t}\frac{\,\rm{R}^{\Sigma_{t}}}{2} \, d\sigma +\int\limits_{\Sigma_{t}}\bigg[\frac{\vert \nabla^{\Sigma_t}\vert \nabla u\vert\vert^{2}}{\vert \nabla u \vert^{2}}+\frac{\RRR}{2}+\frac{\,\vert \ringg{\mathrm{h}}\vert^{2}}{2}+ \frac{3}{4}\left(\frac{4u}{1-u^{2}}\,\,\vert \nabla u \vert-\mathrm{H}\right)^{\!2}\bigg]d\sigma\,.
\end{equation}
Indeed,
\begin{align}
\frac{d}{dt}\int\limits_{\Sigma_{t}} \vert \nabla u \vert^{2}\, d\sigma&=\,-\,\frac{\mathcal{C}}{t^{2}}\left(1+\frac{\mathcal{C}}{2t}\right)^{\!-2}\int\limits_{\Sigma_{t}} \vert \nabla u \vert\,\mathrm{H}\, d\sigma\,,\label{cap4geq18}\\
\frac{d}{dt}\int\limits_{\Sigma_{t}} \vert \nabla u \vert \,\mathrm{H}\,d\sigma&=-\frac{\mathcal{C}}{t^{2}}\left(1+\frac{\mathcal{C}}{2t}\right)^{\!-2} \int\limits_{\Sigma_{t}} \vert \nabla u \vert \left[ \, \Delta^{\Sigma_{t}} \!\left(\frac{1}{\vert \nabla u \vert}\right) + \, \frac{\vert \mathrm{h}\vert^{2}+\Ric (\nu,\nu)}{\vert \nabla u\vert} \, \right]\,d\sigma\,\\
&=-\frac{\mathcal{C}}{t^{2}}\left(1+\frac{\mathcal{C}}{2t}\right)^{\!-2}\int\limits_{\Sigma_{t}} \bigg[ \frac{\vert \nabla^{\Sigma_t}\vert \nabla u\vert\vert^{2}}{\vert \nabla u \vert^{2}}+\frac{\RRR}{2}-\frac{\RRR^{\Sigma_{t}}}{2}+\frac{\,\vert \ringg{\mathrm{h}}\vert^{2}}{2}+\frac{3\,\mathrm{H}^{2}}{4}\bigg]\,d\sigma\,,
\end{align}
where $\na^{\Sigma_{t}}$, $\Delta^{\Sigma_{t}}$ are the Levi--Civita connection and the Laplace--Beltrami operator of the induced metric $g^{\Sigma_{t}}$, respectively, $\rm{R}^{\Sigma_{t}}$ is the scalar curvature of ${\Sigma_{t}}$ and finally, $\mathrm{h}$, $\ringg{\mathrm{h}}$ denote the second fundamental form of ${\Sigma_{t}}$ and its traceless version with respect to $\nu=\na u /\vert\na u\vert$, respectively. Here, the first and second equality follow from the normal first variation of the volume measure and of the mean curvature, whereas the last one is obtained with the help of the traced Gauss equation and the divergence theorem.
Notice that the last integral of the right hand side of equality~\eqref{feq54} is always nonnegative, as $\RRR\geq 0$ (by assumption), and if $\Sigma_{t}$ is connected, then the first two summands also give a nonnegative contribution, by virtue of Gauss--Bonnet theorem, thus $F'(t)\geq 0$.

Let us now prove Proposition~\ref{propomonGcap}, where the difficulty lies in the possible presence of critical values for the function $u$.

\begin{proof}
We consider on $M\setminus \mathrm{Crit}(u)$ the vector field $X$, given by
\begin{align}\label{X}
X:=\frac{1+u}{2(1-u)}\, \nabla u+ \frac{\mathcal{C}}{(1-u)^{2}} \,\nabla \vert \nabla u\vert+ \frac{2\mathcal{C}(2u-1)}{(1+u)(1-u)^{3}}\,\vert \nabla u \vert \nabla u\,.
\end{align}
With the help of Bochner formula,
\begin{equation}\label{Bochnerformulacap4}
\frac{1}{2}\, \Delta\, | \nabla f |^{2}=|  \nabla df  |^{2}+\mathrm{Ric}( \nabla f, \nabla f )+ g(\nabla\,\Delta f, \nabla f)\,
\end{equation}
for every $f\in C^{\infty}(M)$, and being $u$ a harmonic function, the divergence of $X$ can be expressed as
\begin{align}
\label{div(X)nogeom}
\mathrm{div}(X)=\frac{\mathcal{C}\vert \nabla u \vert}{(1-u)^{2}}&\left[\frac{\vert \nabla u \vert}{\mathcal{C}}+\frac{12 u^{2}}{(1-u^{2})^{2}}\,\vert \nabla u \vert^{2}
+\frac{6 u }{ 1-u^{2}} \, \frac{g(\nabla \vert \nabla u\vert,\nabla u )}{\vert \nabla u\vert}\right.\\
&\quad \left. +\,\frac{ \vert \nabla du\vert^{2}-\vert \nabla\vert \nabla u\vert\vert^{2}+\Ric(\na u,\na u)}{{\vert\na u\vert^2}} \right]\,.
\end{align}
Notice that an equivalent expression for $\mathrm{div}(X)$, adapted to the (regular portions of the) level sets of $u$, namely 
\begin{equation}\label{div(X)geom}
\mathrm{div}(X)=\frac{\mathcal{C}\vert \nabla u \vert}{(1-u)^{2}} \bigg[\,\frac{\vert \nabla u \vert}{\mathcal{C}}-\frac{\,\rm{R}^{\Sigma}}{2}+\frac{\vert\nabla^{\Sigma}\vert \nabla u\vert \vert^{2}}{\vert \nabla u \vert^{2}}+\frac{\RRR}{2}+\frac{\, \vert \ringg{\mathrm{h}}\vert^{2}}{2}+\frac{3}{4}\left( \frac{4 u}{1-u^2} \,\vert \nabla u \vert-\mathrm{H}\right)^{\!\!2}\bigg]\,,
\end{equation}
can be obtained by the traced Gauss equation and the identity
\begin{equation} \label{feq70}
\vert  \nabla du \vert^{2}=\vert \nabla u \vert^{2} \vert  \rm{h} \vert^{2}+\vert\nabla\vert \nabla u \vert\vert^{2}+\vert\nabla^{\Sigma}\vert \nabla u \vert\vert^{2}\,.
\end{equation}
Above, $\mathrm{h}, \mathrm{H}, \rm{R}^{\Sigma}$ and $\na^\Sigma$ are all the ones associated to the (regular portions of the) level set of $u$ that passes for the point where 
$\mathrm{div}(X)$ is computed.\\
Let us now show that $F(t)\leq F(T)$ whenever $t,T\in \mathcal{T}$ satisfy $t<T$.
To simplify the exposition, we introduce the diffeomorphism $f:[\mathcal{C}/2, + \infty)\to[0,1)$, defined by
\begin{equation}\label{cap4f}
f(t):=\Big(1-\frac{\mathcal{C}}{2t}\Big)/\Big(1+\frac{\mathcal{C}}{2t}\Big)\,.
\end{equation}
We treat the non--trivial case in which the open interval $\big(f(t),f(T)\big)$ contains critical values of $u$.
In this case, the vector field $X$ is no longer well defined in $\{f(t)\leq u\leq f(T)\}$ and to overcome this difficulty, we consider a pointwise nondecreasing sequence of cut--off functions $\{ \eta_k\}_{k \in \mathbb{N}^{+}}$ such that, for every $k \in \mathbb{N}^{+}$, the functions $\eta_k:[0,+\infty) \to [0,1]$ are smooth, nondecreasing and satisfy
\begin{align}
\label{defcutoffcap4}
\eta_k (\tau) \equiv 0\quad \text{in $\bigg[0 \,,\frac{1}{2k}\bigg]$}\,,\quad 
0\leq \eta_k'(\tau)\leq 2 k\quad \text{in $\bigg[\frac{1}{2k}\,,\frac{3}{2k}\bigg]$},
\quad
\eta_k(\tau) \equiv 1\quad\text{in $\bigg[\frac{3}{2k} \, ,+\infty\!\bigg)$}\,.
\end{align}
Using these cut--off functions, we define for every $k \in\mathbb{N}^{+}$, the vector fields
\begin{equation}
X_{k}:=\frac{1+u}{2(1-u)}\, \nabla u+ \eta_{{k}}\!\left(\frac{\vert \nabla u\vert}{(1-u)(1+u)^{3}}\right)\!\left[ \frac{\mathcal{C}}{(1-u)^{2}}\, \nabla \vert \nabla u\vert+ \frac{2\mathcal{C}(2u-1)}{(1+u)(1-u)^{3}}\,\vert \nabla u \vert \nabla u \right].
\end{equation}
Notice that the vector fields $X_k$ are well defined in $M$ and they coincide with the vector field $X$ in formula~\eqref{X}, whenever restricted to a compact subset of $M\setminus \mathrm{Crit}(u)$, for $k$ large enough. 
Moreover, they have divergence given by the following formula,
\begin{align}
\mathrm{div}(X_{k})
&=\frac{\mathcal{C}\vert \nabla u \vert}{(1-u)^{2}}\left\{\eta_{{k}}\!\left(\frac{\vert  \nabla u \vert}{(1-u)(1+u)^{3}}\right)\left[\frac{6 u }{ 1-u^{2}} \, \frac{g(\nabla \vert \nabla u\vert,\nabla u )}{\vert \nabla u\vert}+\frac{\Ric(\na u,\na u)}{{\vert\na u\vert^2}}\right]\right.\\
&\quad\left.+\eta_{{k}}\!\left(\frac{\vert  \nabla u\vert}{(1-u)(1+u)^{3}}\!\right)\left[\frac{12 u^{2}}{(1-u^{2})^{2}}\, \vert \nabla u \vert^{2}+\frac{ \vert \nabla du\vert^{2}-\vert \nabla\vert \nabla u\vert\vert^{2}}{{\vert\na u\vert^2}}\right] +\frac{\vert \nabla u \vert}{\mathcal{C}}\right\}\\
&\quad+\frac{\mathcal{C}}{(1-u^2)^{3}}\, \eta'_{{k}}\!\left(\frac{\vert \nabla u \vert}{(1-u)(1+u)^{3}}\!\right)\left\vert  \frac{2(2u-1)}{1-u^{2}} \,\vert \nabla u \vert \na u+ \nabla\vert  \nabla u\vert \right\vert^{2} \,. \label{feq50c4}
\end{align}
Since the last summand of the above expression is nonnegative and since, for large enough $k$, the vector field $X_k$ 
coincides with $X$ on the boundary of $\{f(t)< u< f(T)\}$, the divergence theorem, applied to $X_{k}$ on $\{f(t)< u< f(T)\}$, gives 
\begin{align}
\label{feq51c4}
&F(T)-F(t)=\!\!\!\int\limits_{\left\{f(t)< u<f(T)\right\}} \!\!\!\!\!\mathrm{div}(X_k) \, d\mu 
\geq \!\!\! \int\limits_{\left\{f(t)< u<f(T)\right\}} \!\!\!\!\! \widehat{P}_k \, d\mu +\!\!\! \int\limits_{\left\{f(t)< u<f(T)\right\}} \!\!\!\!\! \widehat{D}_k \, d\mu \,,\quad\,\,
\end{align}
where we set 
\begin{align}
\widehat{P}_k&:=\frac{\mathcal{C}\vert \nabla u \vert}{(1-u)^{2}}\left[\frac{\vert \nabla u \vert}{\mathcal{C}}+\eta_{{k}}\!\left(\!\frac{\vert  \nabla u\vert}{(1-u)(1+u)^{3}}\!\right) \widehat{P}\right]\,,\\
\widehat{D}_k&:=\frac{\mathcal{C}\vert \nabla u \vert}{(1-u)^{2}}\, \eta_{{k}}\!\left(\frac{\vert  \nabla u \vert}{(1-u)(1+u)^{3}}\right)\widehat{D}\,,
\end{align}
with
\begin{align}
\widehat{P}&:=\frac{12 u^{2}}{(1-u^{2})^{2}}\, \vert \nabla u \vert^{2} +\frac{ \vert \nabla du\vert^{2}-\vert \nabla\vert \nabla u\vert\vert^{2}}{{\vert\na u\vert^2}}\,,\\
\widehat{D}&:=\frac{6 u }{ 1-u^{2}} \, \frac{g(\nabla \vert \nabla u\vert,\nabla u )}{\vert \nabla u\vert}+\frac{\Ric(\na u,\na u)}{{\vert\na u\vert^2}}\,.
\end{align}
Notice that the functions $\widehat{P}$ and $\widehat{D}$ are $\mu$--a.e. well defined and smooth as well as $\mathrm{div}(X)$, being $\mu\big(\mathrm{Crit}(u)\big)=0$ (see Section~\ref{settmasscapacity}). Furthermore,
\begin{enumerate}[ label=$\mathrm{(\arabic*)}$]
\item one has
$$
0\leq P_k \nearrow\, \frac{\mathcal{C}\vert \nabla u \vert}{(1-u)^{2}}\left[\frac{\vert \nabla u \vert}{\mathcal{C}}+ \widehat{P}\,\mathbb{I}_{M\setminus\mathrm{Crit}(u)}\right]\,\,\,\text{pointwise on $M$ as}\,\,\,k\to+\infty\,,
$$
where $\mathbb{I}_{M\setminus\mathrm{Crit}(u)}$ denotes the characteristic function of $M\setminus\mathrm{Crit}(u)$;
\item one gets
$$
\widehat{D}_{k}\to \frac{\mathcal{C}\vert \nabla u \vert}{(1-u)^{2}}\,\widehat{D}\, \mathbb{I}_{M\setminus\mathrm{Crit}(u)}\,\,\,\,\,\text{pointwise on $M$ as}\,\,\,k\to+\infty\,,
$$
where there hold $\mu$--a.e. the following inequalities
$$
\vert  \widehat{D}_k  \vert \leq\frac{\mathcal{C}\vert \nabla u \vert}{(1-u)^{2}}\, \vert  \widehat{D}\vert \quad\quad \text{and}\quad\quad\vert  \widehat{D}\vert \leq\bigg[\frac{6 u }{ 1-u^{2}}\, \vert \nabla du\vert +\vert\Ric\vert\,\bigg] \in L^{1}_{loc}(M)\,;
$$
\item finally, one observes
$$
\mathrm{div}(X)=\frac{\mathcal{C}\vert \nabla u \vert}{(1-u)^{2}}\bigg[\frac{\vert \nabla u \vert}{\mathcal{C}}+\widehat{P}+\widehat{D}\bigg]
$$
in $M\setminus\mathrm{Crit}(u)$, by identity~\eqref{div(X)nogeom}.
\end{enumerate}
By point $\mathrm{(2)}$, the dominated convergence theorem implies
\begin{align}\label{cap4geq15}
\lim_{k\to + \infty}\int\limits_{\left\{f(t)< u<f(T)\right\}} \!\!\! \widehat{D}_k \, d\mu=\!\!\!\int\limits_{\left\{f(t)< u<f(T)\right\}} \!\!\!\frac{\mathcal{C}\vert \nabla u \vert}{(1-u)^{2}}\,\widehat{D} \, d\mu \,,
\end{align}
whereas, by point $\mathrm{(1)}$ and the monotone convergence theorem, it follows
\begin{align}\label{cap4geq16}
\lim_{k\to + \infty}\int\limits_{\left\{f(t)< u<f(T)\right\}} \!\!\! \widehat{P}_k \, d\mu=\!\!\!\int\limits_{\left\{f(t)< u<f(T)\right\}} \!\!\!\frac{\mathcal{C}\vert \nabla u \vert}{(1-u)^{2}}\bigg[\frac{\vert \nabla u \vert}{\mathcal{C}}+ \widehat{P}\,\bigg]\, d\mu \,.\end{align}
As a consequence of inequality~\eqref{feq51c4} with the existence of limit~\eqref{cap4geq15} finite, the sequence of nonnegative real numbers given by the integrals of the functions $P_k$ is bounded from above, then
$$
\frac{\mathcal{C}\vert \nabla u \vert}{(1-u)^{2}}\bigg[\frac{\vert \nabla u \vert}{\mathcal{C}}+ \widehat{P}\,\bigg]\in L^{1}\big(\left\{f(t)< u<f(T)\right\}\big)\,.
$$
Thus, passing to the limit, as $k \to + \infty$, in inequality~\eqref{feq51c4}, by limits~\eqref{cap4geq15} and~\eqref{cap4geq16} together with point $\mathrm{(3)}$ above, we get
\begin{align}
F(T)-F(t)&\geq \!\!\!\int\limits_{\left\{f(t)< u<f(T)\right\}} \!\!\!\!\mathrm{div}(X) \, d\mu \,\\
&=\!\!\!\int\limits_{\left[f(t),f(T)\right]\setminus \, \mathcal{N}}\!\!\!\!\!\!\frac{\mathcal{C}\,ds}{(1-s)^{2}}\,\Bigg\{\int\limits_{\{u=s\}}\!\!\!\bigg[ \frac{\vert\nabla^{\Sigma}\vert \nabla u\vert \vert^{2}}{\vert \nabla u \vert^{2}}+\frac{\RRR}{2}+\frac{ \vert  \ringg{\mathrm{h}} \vert^{2}}{2} \bigg]\, d\sigma \\
&\phantom{=\!\!\!\int\limits_{\left[f(t),f(T)\right]\setminus \, \mathcal{N}}\!\!\!\!\!\!ds\,\frac{\mathcal{C}}{(1-s)^{2}}\,\Bigg\{}\quad+ \frac{3}{4}\!\!\int\limits_{\{u=s\}}\!\!\!\!\left( \frac{4 u}{1-u^2} \, \vert \nabla u \vert-\mathrm{H}\right)^{\!\!2}\!d\sigma\\
&\phantom{=\!\!\!\int\limits_{\left[f(t),f(T)\right]\setminus \, \mathcal{N}}\!\!\!\!\!\!ds\,\frac{\mathcal{C}}{(1-s)^{2}}\,\Bigg\{}\quad+4\pi -\!\!\!\!\int\limits_{\{u=s\}}\!\!\!\! \frac{\,\,\rm{R}^{\Sigma}}{2} \, d\sigma \Bigg\}\,,
\end{align}
where $\mathcal{N}$ is the set of the critical values of $u$.
Here, the equality follows first by using the coarea formula, then by applying equality~\eqref{div(X)geom} for the divergence of $X$ and finally by Sard theorem.
Since we are integrating only along the regular level sets of $u$ and since every regular level set is a connected (by assumption) closed surface, we can invoke Gauss--Bonnet theorem to deduce that the last two summands also give a nonnegative contribute, while the other summands are always nonnegative, as $\RRR\geq 0$ (by assumption).
The claimed monotonicity of $F$ hence follows. 
\end{proof}

\begin{proposition}[Rigidity -- I]\label{Rigstat1}
Let $(M,g)$ be a $3$--dimensional, complete, one--ended asymptotically flat manifold with compact, connected boundary and with nonnegative scalar curvature.
Let $u\in C^{\infty}(M)$ be the solution of Dirichlet problem~\eqref{f1prelc4} and let $\mathcal{C}>0$ be the boundary capacity of $\partial M$ in $(M,g)$, given by formula~\eqref{boundcapintermofu}. Consider the function $F:[\mathcal{C}/2, + \infty)\to \R$ defined by equality ~\eqref{feq40}.
Then, $F$ is constant on the set $\mathcal{T}$, given by equality~\eqref{defTcap}, if and only if $(M,g)$ is isometric to a (exterior spatial) Schwarzschild manifold $(M_{\mathrm{Sch}(m)},\,g_{\mathrm{Sch}(m)})$ of mass $m>0$, see formula~\eqref{feq63bisbis}.
\end{proposition}
\begin{proof}
If $(M,g)$ is the Schwarzschild manifold $(M_{\mathrm{Sch}(m)},\,g_{\mathrm{Sch}(m)})$ with mass $m>0$, 
\begin{equation}\label{heq1}
u= \frac{1-\frac{m}{2\vert  x \vert}}{1+\frac{m}{2\vert  x \vert} }\,,\quad\quad\quad \vert \nabla u\vert=\Big(1+\frac{m}{2\vert  x \vert}\Big)^{\!-4}\frac{m}{\vert  x \vert^{2}},\quad \quad\quad \mathrm{H}=\frac{2}{\vert  x \vert}\,\frac{1-\frac{m}{2\vert  x \vert}}{\big( 1+\frac{m}{2\vert  x \vert}\big)^{3}}\,. 
\end{equation}
Notice that $u$ has no critical points.
By a straightforward computation, one has
$$
\mathcal{C}:=\frac{1}{4\pi}\int\limits_{\partial M}\vert \na u \vert \,d\sigma=m\qquad\text{ and }\qquad F\equiv0\,.
$$
Now, we assume that $F$ is constant on the set $\mathcal{T}$.
We know that there exists a maximal time $T\in(\mathcal{C}/2, + \infty]$ such that $\na u \neq 0$ in $M_{T}:=\big\{0\leq u<(1-\frac{\mathcal{C}}{2T})/(1+\frac{\mathcal{C}}{2T})\big\}$, since $\mathcal{T}\supseteq[a, b)$ with $a=\mathcal{C}/2$. 
Then, $\left[\frac{\mathcal{C}}{2},T\right)\subseteq\mathcal{T}$ and $F$ is continuously differentiable in $\left[\frac{\mathcal{C}}{2},T\right)$, with $F'$ given by formula~\eqref{feq54}. 
Observe that for every $t\in\left[\frac{\mathcal{C}}{2},T\right)$, each level set $\{u=(1-\frac{\mathcal{C}}{2t})/(1+\frac{\mathcal{C}}{2t})\}$, being diffeomorphic to the boundary $\partial M$ (as $\partial M=\{u=0\}$), is a connected and closed surface,
then, $F'(t)\geq 0$, as explained before the proof of Proposition~\ref{propomonGcap}, but at the same time $F'(t)=0$, by assumption.
Consequently, all the nonnegative summands in formula~\eqref{feq54} are forced to vanish for every $t \in [\mathcal{C}/2,T)$. This fact gives $\na^{\Sigma_t}|\na u| =\na^{\top}|\na u| = 0$ and $\HHH=\frac{4u}{1-u^{2}}\,|\na u|$, which in turn imply
\begin{equation}\label{cap4geq23}
\na|\na u|=\na^{\perp}|\na u|
=-\HHH\,\na u=-\frac{4u}{1-u^{2}}\,\vert \nabla u \vert\,\na u\,.
\end{equation}
Then,
\begin{align*}
\na\bigg[\!\log\bigg(\frac{\vert \nabla u \vert}{\left(1-u^{2}\right)^{2}}\bigg)\bigg]&=0\,,
\end{align*}
hence, the function $\vert \nabla u \vert/\left(1-u^{2}\right)^{2}$ is constant on every connected component of $M_{T}$, but $M_{T}$ is connected since it is diffeomorphic
to $\left[0,\,(1-\frac{\mathcal{C}}{2T})/(1+\frac{\mathcal{C}}{2T})\right)\times\partial M$ and $\partial M$ is connected.
In conclusion, $|\na u|=a (1-u^{2})^{2}$, where $a\in\R$ is a positive constant, therefore, being $0\leq u<1$ on the whole manifold, $T= +\infty$ and $|\na u| \neq 0$ everywhere. 
Then, $F$ is of class $C^{1}$ on $[\mathcal{C}/2,+\infty)$ and all the level sets of $u$ are regular and diffeomorphic to each other, which clearly implies that they are all connected, hence $F'(t)$ can be written as sum of nonnegative terms, which are forced to vanish, as $F$ is constant, as before.
Concerning the constant $a$, from formulas~\eqref{hexpu} and~\eqref{feq48}, it follows
$$
\mathcal{C}=\lim\limits_{\vert  x\vert\to +\infty}\vert  x\vert^{2} |\na u|= a \lim\limits_{\vert  x\vert\to +\infty}\vert  x\vert^{2}\, (1-u^{2})^{2}=4 a\mathcal{C}^{2}\,,
$$
as a result, $a=(4\mathcal{C})^{-1}$.
Now, up to an isometry, we have that $M=[0,1)\times\partial M$, every slice $\{t\}\times\partial M$ is the level set $\{u=t\}$ and the metric $g$ can be written as
$$
g = \frac{(4\mathcal{C})^{2}}{(1-u^{2})^4} \,du \otimes du +g_{\alpha\beta} (u, \!\vartheta) \, d\vartheta^{\alpha} \otimes d\vartheta^\beta \, ,
$$ 
where $g_{\alpha\beta}(u, \!\vartheta) \, d\vartheta^\alpha \otimes d\vartheta^\beta$ represents the metric induced by $g$ on the level sets of $u$. By the vanishing of the traceless second fundamental form of the level sets in formula~\eqref{feq54}, i.e. $ \mathrm{h}_{\alpha\beta}=(\HHH/2)\, g_{\alpha\beta}$, in combination with equality $\mathrm{h}_{\alpha\beta}=\nabla du_{\alpha\beta}/|\na u|$, it turns out that the coefficients $g_{\alpha\beta}(u,\!\vartheta)$ satisfy the following first order system of PDE's
$$
\frac{\pa g_{\alpha\beta}}{\pa u} = \frac{4u }{1-u^{2}}\, g_{\alpha\beta}\, ,
$$
from which we can deduce
$$
g_{\alpha\beta}(u, \!\vartheta) \, d\vartheta^\alpha \otimes d\vartheta^\beta= (1-u^{2})^{-2} c_{\alpha\beta}(\vartheta)\, d\vartheta^\alpha \otimes d\vartheta^\beta\,.
$$
At the same time, for every $u_{0}\in [0,1)$, we also have
\begin{equation}\label{feq71}
\frac{1}{2}\,\mathrm{R}^{\{u=u_{0}\}}=\frac{(1-u_{0}^{2})^{2}}{4\mathcal{C}^{2}}\,,
\end{equation}
indeed, from the traced Gauss equation together with Bochner formula~\eqref{Bochnerformulacap4} (coupled with the fact that $u$ is harmonic), it follows
\begin{equation}
\mathrm{R}^{\{u=u_{0}\}}=\mathrm{R}+\vert \nabla u\vert^{-2}\left[ -\Delta\vert \nabla u\vert^{2} +2\,\vert \nabla du \vert^{2}\,\right]-\vert \mathrm{h}\vert^{2}+\HHH^{2}\,.
\end{equation}
This equality, by using the vanishing of the scalar curvature of $M$ and of the traceless second fundamental form of the level sets in equality~\eqref{feq54}, along with the consequence of the vanishing of this latter in identity~\eqref{feq70} (take into account also formula~\eqref{cap4geq23}) given by
\begin{equation}
\vert \nabla du \vert^{2}=\frac{3}{2}\,\vert \nabla u \vert ^{2}\,\HHH^{2}\,,
\end{equation}
becomes
\begin{align}
\mathrm{R}^{\{u=u_{0}\}}&=-\vert \nabla u\vert^{-2}\,\Delta\vert \nabla u\vert^{2}+\frac{7\,\HHH^{2}}{2}\,,
\end{align}
but, being
\begin{equation}
\qquad|\na u|=a (1-u^{2})^{2}\qquad\text{and}\qquad \HHH=\frac{4u}{1-u^{2}}\,|\na u|=4a u (1-u^{2})\,,
\end{equation}
with $a=(4\mathcal{C})^{-1}$, as already explained, one obtains identity~\eqref{feq71}.
Then, $\partial M$ has constant sectional curvature (equal to $1/(4\mathcal{C}^{2})\,$), with the Riemannian metric induced by $(M,g)$, and it is diffeomorphic to a $2$--sphere, see Section~\ref{settmasscapacity}.
Consequently, $(\partial M,\,g_{\partial M})$ is isometric to $(\SSS^{2},4\mathcal{C}^{2}\,g_{\mathbb{S}^2})$, by~\cite[Section~3.F]{gahula}, thus, up to an isometry, one has $M=[0,1)\times \SSS^{2}$ and 
\begin{align}
g&= \frac{(4\mathcal{C})^{2}}{(1-u^{2})^4} \,du \otimes du +\frac{4\mathcal{C}^{2}}{(1-u^{2})^{2}}\,g_{\mathbb{S}^2}\,,
\end{align}
therefore, the map
$$
(u,\vartheta)\in\left(M,\,g\right)\mapsto \frac{\mathcal{C}}{2}\,\frac{1+u}{1-u}\,\vartheta\in\left(M_{\mathrm{Sch}(\mathcal{C})},\,g_{\mathrm{Sch}(\mathcal{C})}\right)
$$
is an isometry, see formula~\eqref{feq63bisbis}. 
\end{proof}

The natural next step, after the monotonicity formula has been proven and the case of when it is constant has been studied, is to compute and then to compare the limit of the function $F(t)$ as $t \to+\infty$  with its value at $t = \mathcal{C}/2$,
indeed, $\partial M$ is a regular value of $u$ and the existence of the limit is a consequence of the fact that $F$ is nondecreasing in a neighborhood of $+\infty$, because the level sets $\Sigma_{t}$, given by equality~\eqref{Sigmat}, are regular (that is $t\in\mathcal{T}$) and connected (namely, they are diffeomorphic to a $2$--sphere) for each $t$ sufficiently large, see the notes before the proof of Proposition~\ref{propomonGcap} and Section~\ref{settmasscapacity}.
These computations are contained in the following lemma.

\begin{lemma}\label{Fagliestremi}
Let $(M,g)$ be a $3$--dimensional, complete, one--ended asymptotically flat manifold with compact, connected boundary and with nonnegative scalar curvature.
Let $u\in C^{\infty}(M)$ be the solution of Dirichlet problem~\eqref{f1prelc4} and let $\mathcal{C}>0$ be the boundary capacity of $\partial M$ in $(M,g)$, given by formula~\eqref{boundcapintermofu}. Consider the function $F:[\mathcal{C}/2, + \infty)\to \R$ defined by equality~\eqref{feq40}.
Then, there hold
\begin{align}
F\left(\mathcal{C}/2\right)&=\mathcal{C}\Bigg[2\pi-2\!\int\limits_{\partial M}\!\!\vert \nabla u \vert^{2}\, d\sigma-\!\int\limits_{\partial M}\!\!\vert \nabla u \vert\,\mathrm{H}\,d\sigma\Bigg]\,,\label{feq41}\\
\lim_{t\to+\infty}F(t)&\leq8\pi\left(m_{\mathrm{ADM}}-\mathcal{C}\right)\,.\label{feq42}
\end{align}
\end{lemma}

\begin{proof}
The function $F$ is easily seen to satisfy equality~\eqref{feq41}, now, we check formula~\eqref{feq42}.
As already explained above, there exists $t_{0}\in \left[\mathcal{C}/2,+\infty\right)$ such that $[t_{0},+\infty)\subseteq \mathcal{T}$, therefore, from now on the variable $t$ ranges in $[t_{0},+\infty)$.
We break $F$ in two pieces, 
\begin{align}
F_{1}(t)&=\,\,4\pi t\,\,+\,\,\, \frac{t^{3}}{\mathcal{C}^{2}} \,\left(1+\frac{\mathcal{C}}{2t}\right)^{\!3}\int\limits_{\Sigma_{t}}\vert \nabla u \vert^{2}\, d\sigma\,\,-\,\,\frac{t^{2}}{\mathcal{C}} \,\left(1+\frac{\mathcal{C}}{2t}\right)^{\!2}\int\limits_{\Sigma_{t}}\vert \nabla u \vert\,\mathrm{H}\, d\sigma\,,\quad\,\,\,\label{cap4geq1}\\
F_{2}(t)&=\,-\,\frac{3t^{2}}{2\mathcal{C}} \,\left(1+\frac{\mathcal{C}}{2t}\right)^{\!3} \int\limits_{\Sigma_{t}}\vert \nabla u \vert^{2}\, d\sigma\,,
\end{align}
where the function $F_{2}$ satisfies 
\begin{equation}\label{feq56}
\lim_{t\to+\infty}F_{2}(t)=-6\pi\mathcal{C}\,.
\end{equation}
Indeed, once the function $F_{2}$ is rewritten as
\begin{align}\label{feq57}
F_{2}(t)&=\,-\,\frac{3\mathcal{C}}{2} \,\left(1+\frac{\mathcal{C}}{2t}\right) \int\limits_{\Sigma_{t}}\frac{\vert \nabla u \vert^{2}}{(1-u)^{2}}\, d\sigma\,,
\end{align}
one observes 
\begin{equation}\label{feq48}
\vert \nabla u\vert=\frac{\mathcal{C}}{\,\,\vert x\vert^{2}}\left[\,1+O\big(\vert x\vert^{-\beta}\big)\,\right]\,,
\end{equation}
from the expansion~\eqref{hexpu} of $u$ in an asymptotically flat chart, fixed from now on, thus,
\begin{equation}\label{feq55}
\lim_{\vert x\vert\to + \infty}\frac{\vert \nabla u \vert}{(1-u)^{2}}=\mathcal{C}^{-1} \,.
\end{equation}
Then, using this fact in combination with formula~\eqref{feq44}, one obtains that the integral term in expression~\eqref{feq57} of $F_{2}(t)$ converges to $4\pi$, as $t\to+\infty$, consequently, the limit~\eqref{feq56} holds.\\
Regarding the function $F_{1}$, we introduce the auxiliary function $\rho:M\to[\mathcal{C}/2,+\infty)$, given as
\begin{equation}
\rho:=\frac{\mathcal{C}}{2}\,\frac{1+u}{1-u}\,
\end{equation}
and called {\em Euclidean fake distance}, due to the fact that
\begin{equation}\label{fakedistance}
\rho=\vert x\vert+O_{2}(\vert x\vert^{1-\beta})\,.
\end{equation}
Observe that $\Sigma_{t}=\{\rho=t\}$ and the function $F_{1}$, expressed in terms of $\rho$ as following
\begin{align*}
F_{1}(t)&=\frac{t}{4}\,\left(1+\frac{\mathcal{C}}{2t}\right)\Bigg\{16\pi\left(1+\frac{\mathcal{C}}{2t}\right)^{\!-1}\!\!-\!\int\limits_{\{\rho=t\}}\!\!\!\mathrm{H}^{2}\,d\sigma+\!\int\limits_{\{\rho=t\}}\!\!\!
\left[\frac{g(\nabla \vert \nabla \rho\vert , \nabla \rho)}{ \vert \nabla \rho\vert^{2}}\right]^{2}\!d\sigma\Bigg\}\,,
\end{align*}
can then be broken in three pieces, 
\begin{align}
F_{11}(t)&=-2\pi\mathcal{C}\,,\\
F_{12}(t)&=\frac{t}{4}\,\left(1+\frac{\mathcal{C}}{2t}\right)\bigg[16\pi\,-\!\int\limits_{\{\rho=t\}}\!\!\!\mathrm{H}^{2}\,d\sigma\bigg]\,,\quad\,\,\,\label{cap4geq2}\\
F_{13}(t)&=\frac{t}{4}\,\left(1+\frac{\mathcal{C}}{2t}\right)\int\limits_{\{\rho=t\}}\!\!\!\left[\frac{g(\nabla \vert \nabla \rho\vert , \nabla \rho)}{ \vert \nabla \rho\vert^{2}}\right]^{2}\!\!\!d\sigma\,=\,\frac{1}{4}\,\left(1+\frac{\mathcal{C}}{2t}\right)\int\limits_{\{\rho=t\}}\!\!\!\rho\!\left[\frac{\nabla d\rho(\nabla \rho , \nabla \rho)}{ \vert \nabla \rho\vert^{3}}\right]^{2}\!\!\!d\sigma\,.\label{cap4geq3}
\end{align}
By expansion~\eqref{hexpu} of $u$, there exist positive constants $A_{1},A_{2},B_{1},B_{2}$ such that
\begin{equation}\label{cap4geq4}
\frac {A_{1}}{\vert x\vert}\,\leq\,1-u\,\leq\,\frac {A_{2}}{\vert x \vert}\,\qquad\text{and}\qquad \frac {B_{1}}{\vert x\vert^{2}}\,\leq\,\vert  \nabla u \vert \,\leq\,\frac {B_{2}}{\vert x \vert^{2}}\,.
\end{equation}
for $\vert  x \vert$ large enough. 
Then, for $t$ sufficiently large, one has
\begin{equation}\label{eqf45}
\frac{A_{1} t}{\mathcal{C}}\leq\vert x(p)\vert\leq \frac{2A_{2} t}{\mathcal{C}}
\end{equation}
for every $q\in \Sigma_{t}$, consequently, integrating $|\na u|$ on $\Sigma_t$, by equality~\eqref{feq44} one gets
\begin{equation}\label{eqf46}
\frac{4\pi A_{1}^{2}}{B_{2}\mathcal{C}} \,t^2 \,\,\leq \,\mathrm{Area}(\Sigma_t) \,\leq \frac{16\pi A_{2}^{2}}{B_{1}\mathcal{C}}\,t^2 \,.
\end{equation}
Hence, we have
\begin{equation}
0\leq\int\limits_{\{\rho=t\}}\!\!\!\rho\!\left[\frac{\nabla d\rho(\nabla \rho , \nabla \rho)}{ \vert \nabla \rho\vert^{3}}\right]^{2}\!d\sigma\leq C t^{1-2\beta}\,
\end{equation}
for $t$ large enough, since
\begin{equation}
\frac{\nabla d\rho(\nabla \rho , \nabla \rho)}{ \vert \nabla \rho\vert^{3}}=O(\vert x\vert^{-1-\beta})\,,
\end{equation}
as a consequence of the behavior near infinity of $\rho$, described by formula~\eqref{fakedistance}.
Then, being $\beta>1/2$, one gets the convergence of $F_{13}(t)$ to zero, as $t\to+\infty$.\\
We remark that a key point is the knowledge that the error term in formula~\eqref{hexpu} is $O_{2}(\vert x\vert^{-1-\beta})$, with $\beta>1/2$. Indeed, if this error term was only $o_{2}(\vert x\vert^{-1})$, then the limit of $F_{13}(t)$ would be a indeterminate form.\\
Concerning the function $F_{12}$, observe that
\begin{equation}
\lim_{t\to +\infty}F_{12}(t)\leq\limsup_{t\to +\infty}\,\frac{t}{4} \bigg[16\pi\,-\!\int\limits_{\{\rho=t\}}\!\!\!\mathrm{H}^{2}\,d\sigma\bigg]\,,
\end{equation}
therefore, in the same spirit as in~\cite{HI} and similarly to~\cite[Lemma 2.5]{AMMO}, we are going to show for completeness that the upper limit to the right hand side of the inequality above is less or equal than $8\pi m_{\mathrm{ADM}}$.\\
Our first aim is to obtain an expansion for the mean curvature of $\Sigma_t$ in terms of its Euclidean mean curvature.
In order to do this, it is convenient to write the subscript $g$ for the quantities that are referred to the original metric, and without subscripts the quantities that are referred to the Euclidean metric. 
We underline that this convention will be followed only in this proof.
Finally, the Levi--Civita connection with respect to $g$ will continue to be denoted by $\na$, whereas the symbol $\D$ will indicate the Euclidean one.
As the unit normal vectors to a regular level set $\Sigma_t$ are given by
$$
\nu_g \,= \,\frac{\na \rho \,\,}{|\na \rho|_g} \qquad \hbox{and} \qquad \nu \,= \,\frac{\D \rho}{|\D \rho|} \,, 
$$
the mean curvatures are  
\begin{equation}\label{eqmeancurvatures}
\HHH_g \,= \,\left( g^{ij} - \nu_g^i \nu_g^j \right) \,\frac{\na_i \na_j \rho}{|\na \rho|_g} \qquad \hbox{and} \qquad \HHH \,= \,\left( \delta^{ij} - \nu^i \nu^j \right)\frac{ \D_i \D_j \rho}{|\D \rho|} \,, 
\end{equation}
respectively. According to formula~\eqref{eq4}, one has
\begin{equation}\label{exp2}
g^{ij} = \,\delta^{ij}-\gamma^{ij}+O(\vert  x  \vert^{-2\tau})\,
\end{equation}
where $\gamma_{ij}=g_{ij}-\delta_{ij}$ and $\gamma^{ij} = \delta^{i \ell} \delta^{k j} \gamma_{k \ell}$, hence, the $g$--unit normal is related to the Euclidean one through the formula
\begin{equation}
\nu^i_g \,= \,\Big( 1 + \frac{\gamma(\nu,\nu)}{2}\Big) \,\nu^i - \,\gamma^i_k\,\nu^k \,+ O(|x|^{-2 \tau}) \,,
\end{equation}
where $\gamma^i_k = \delta^{i j} \gamma_{j k}$. It follows then
\begin{equation}
g^{ij} - \nu_g^i \nu_g^j \,\,= \,\,\eta^{i j}\,- \,\eta^{i k}\,\gamma_{k \ell} \,\eta^{j \ell}\,+ \,O(|x|^{-2 \tau}) \,,
\end{equation}
where $\eta^{i j} = \delta^{ij} - \nu^i \nu^j $.
A straightforward computation leads to 
\begin{equation}
\frac{\na_i \na_j \rho}{|\na \rho|_g} \,= \,\Big( 1 + \frac{\gamma(\nu,\nu)}{2}\Big) \,\frac{ \D_i \D_j \rho}{|\D \rho|} \,- \,\frac{1}{2} \,\big( \pa_i g_{j \ell} + \pa_j g_{i \ell} - \pa_\ell g_{ij} \big) \,\nu^\ell + \,O(|x|^{-1-2 \tau}) \,,
\end{equation}
where we employ expansion~\eqref{fakedistance}, to ensure that $|\D \D \rho| / |\D \rho| = O (|x|^{-1})$, indeed, 
\begin{equation}\label{eqf47}
\D \D \rho / |\D \rho|=\frac{1}{\vert x\vert} \left[\eta_{ij}+O(\vert x\vert^{-\beta})\right]dx^{i}\otimes dx^{j}\,,
\end{equation}
setting $\eta_{ij}=\delta_{ij}-\delta_{ik}\,\delta_{j \ell}\,\nu^{k}\,\nu^{\ell}\,$. We then arrive at 
\begin{equation}\label{eqHH}
\HHH_g \,= \,\Big( 1 + \frac{\gamma(\nu,\nu)}{2}\Big) \,\HHH - \eta^{ij} \Big( \pa_j g_{ik} - \frac{1}{2} \pa_k g_{ij} \Big) \nu^k - \eta^{ik} \eta^{j \ell} \gamma_{k \ell} \frac{ \D_i \D_j \rho}{|\D \rho|} + O(|x|^{-1-2 \tau}) \,,
\end{equation}
which implies 
$$
\HHH^2_g \,= \,\big( 1 + {\gamma(\nu,\nu)} \big) \,\HHH^2 - 2 \HHH \,\eta^{ij} \Big( \pa_j g_{ik} - \frac{1}{2} \pa_k g_{ij} \Big) \nu^k - 2\HHH \,\eta^{ik} \eta^{j \ell} \gamma_{k \ell} \frac{ \D_i \D_j \rho}{|\D \rho|} + O(|x|^{-2-2 \tau}) \,.
$$
As the metric induced on $\Sigma_{t}$ by $g$ can be written as $g -  d\rho \otimes d\rho/|\na \rho|^2_g$, the area element can be expressed as 
\begin{equation}
d\sigma_g \,= \,\Big[1+\frac{1}{2}\,\eta^{ij}\gamma_{ij}+O(\vert  x  \vert^{-2\tau})\Big] \,d{\sigma}\,. \label{exp4} 
\end{equation}
Putting all together, the Willmore energy integrand then satisfies
\begin{align}
\HHH^2_g \,d\sigma_g\,\,= \,&\,\left[\Big( 1 + \gamma(\nu,\nu) + \frac{\eta^{ij}\gamma_{ij}}{2} \Big) \,\HHH^2 \,\right.\\
&\left. \quad- \,2 \HHH \,\eta^{ij} \Big( \pa_j g_{ik} - \frac{1}{2} \pa_k g_{ij} \Big) \nu^k - \,2\HHH \,\eta^{ik} \eta^{j \ell} \gamma_{k \ell} \frac{ \D_i \D_j \rho}{|\D \rho|} + O(|x|^{-2-2 \tau}) \right] \,d{\sigma}.
\quad\qquad \label{eq:will}
\end{align}
By equalities~\eqref{eqmeancurvatures} and~\eqref{eqf47} and by using again expansion~\eqref{fakedistance}, we obtain
\begin{equation}\label{eqf48}
{\mathrm{H}}\,= \,\frac{2}{\vert x\vert}\bigl(1+O(\vert x\vert^{-\beta})\bigr) 
\end{equation}
which implies
\begin{equation}\label{eql49}
2\HHH \,\eta^{ik} \eta^{j \ell} \gamma_{k \ell} \frac{ \D_i \D_j \rho}{|\D \rho|} \,= \,\frac{4}{|x|^2} \eta^{k \ell} \gamma_{k \ell} + O(|x|^{-2-2\beta})\,.
\end{equation}
Plugging this information in formula~\eqref{eq:will}, we obtain
\begin{equation}
\HHH^2_g \,d\sigma_g\,\,= \,\,\left[ \,\HHH^2 \,+ \,\frac{4}{|x|^2} \gamma(\nu , \nu) \,- \,\frac{2}{|x|^2} \eta^{i j} \gamma_{i j} \,- \,\frac{4}{|x|} \eta^{ij} \Big( \pa_j g_{ik} - \frac{1}{2} \pa_k g_{ij} \Big) \nu^k + O(|x|^{-2-2 \beta}) \right] \,d{\sigma} \,.
\end{equation}
To proceed, we now claim that 
\begin{equation}
\label{claimdiv}
\frac{4}{|x|^2} \gamma(\nu , \nu) \,- \,\frac{2}{|x|^2} \eta^{i j} \gamma_{i j} \,= \,\frac{2}{|x|} \bigl( \eta^{ij} \pa_i g_{jk} \nu^k - \mathrm{div}_{\Sigma_{t}} X^\top\,\bigr) \,+ \,O(|x|^{-2-2\beta}) \,,
\end{equation}
where $X$ is the vector field defined by $X = \delta^{ij}\gamma_{jk} \nu^k \partial_{i}$ and $X^\top$ denotes its tangential component. To prove the claim, let us first observe that
$$
X =  X^\top + \gamma(\nu,\nu)\nu\qquad\text{ and }\qquad
\pa_i \nu^k \,= \,\eta^{k \ell} \frac{ \D_i \D_\ell \rho}{|\D \rho|} \,.
$$
By means of the expansions~\eqref{eqf48} and~\eqref{eql49}, we compute 
\begin{align}
\eta^{ij} \pa_i g_{jk} \nu^k =\eta^{ij} \pa_i \gamma_{jk} \nu^k & = \,\eta^{ij}\,\delta_{jk} \,\pa_i X^{k} \,- \,\eta^{ik} \eta^{j \ell} \gamma_{k \ell} \frac{ \D_i \D_j \rho}{|\D \rho|}\\
& = \,\mathrm{div}\,X \,- \, \delta_{jk}\pa_i X^{k}\nu^i \nu^j \,- \,\frac{1}{|x|} \eta^{ij} \gamma_{ij} + O(|x|^{-1-2\beta})\\
& = \,\mathrm{div}_{\Sigma_t} X \,-\,\frac{1}{|x|} \eta^{ij} \gamma_{ij} + O(|x|^{-1-2\beta})\\
& = \,\mathrm{div}_{\Sigma_t} X^\top \,+ \,\gamma(\nu,\nu) \HHH - \,\frac{1}{|x|} \eta^{ij} \gamma_{ij} + O(|x|^{-1-2\beta})\\
& = \,\mathrm{div}_{\Sigma_t}X^\top \,+ \,\frac{2}{|x|}\gamma(\nu,\nu) - \,\frac{1}{|x|} \eta^{ij} \gamma_{ij} + O(|x|^{-1-2\beta}) \,.
\end{align}
Claim~\eqref{claimdiv} then follows with the help of some simple algebra. As a consequence, the expression for the Willmore energy integrand becomes
\begin{align}
\HHH^2_g \,d\sigma_g\,\,= \,\,& \left[ \,\HHH^2 - \frac{2}{|x|} \mathrm{div}_{\Sigma_{t}} X^\top + \frac{2}{|x|} \eta^{ij} \pa_i g_{jk} \nu^k - \frac{4}{|x|} \eta^{ij} \pa_j g_{ik}\nu^k + \frac{2}{|x|}\eta^{ij} \pa_k g_{ij}\nu^k+ O(|x|^{-2-2 \beta}) \right] d{\sigma}\\
= \,\,& \left[ \,\HHH^2 - \frac{2}{|x|} \mathrm{div}_{\Sigma_{t}} X^\top - \frac{2}{|x|} \delta^{ij} \big( \pa_i g_{jk} - \pa_k g_{ij} \big) \nu^k + O(|x|^{-2-2 \beta}) \right] d{\sigma} \,. \label{eq:nice}
\end{align}
Hence, one gets
\begin{align}
\frac{t}{4}\,\,\bigg( 16\pi \,\,- \!\int\limits_{\{\rho=t\}}\!\!\!\mathrm{H}_g^{2}\,d\sigma_g \bigg)&\,=\,\frac{t}{4}\,\,\bigg( 16\pi \,\,- \!\int\limits_{\{\rho=t\}}\!\!\!\mathrm{H}^{2}\,d\sigma \bigg) \\
&\,+ \,\frac{1}{4} \int\limits_{\{\rho=t\}}\!\!\! \rho\left[\,\frac{2}{|x|} \mathrm{div}_{\Sigma_{t}} X^\top +\frac{2}{|x|} \delta^{ij} \big( \pa_i g_{jk} - \pa_k g_{ij} \big) \nu^k + O(|x|^{-2-2 \beta})\right]\,d\sigma  \,.
\end{align}
By virtue of equation~\eqref{exp4}, the same estimates of formula~\eqref{eqf46} hold for the Euclidean area $\mathrm{Area}(\Sigma_t)$, up to a different choice of the constants. 
In view of this consideration and employing the expansion~\eqref{fakedistance} with formula~\eqref{eqf45}, one has
\begin{align*}
\frac{t}{4}\,\,\bigg( 16\pi \,\,- \!\int\limits_{\{\rho=t\}}\!\!\!\mathrm{H}_g^{2}\,d\sigma_g \bigg)&\,=\,\frac{t}{4}\,\,\bigg( 16\pi \,\,- \!\int\limits_{\{\rho=t\}}\!\!\!\mathrm{H}^{2}\,d\sigma \bigg) \\
&\,+ \,\frac{1}{2} \int\limits_{\{\rho=t\}} \mathrm{div}_{\Sigma_{t}}X^\top\,d\sigma + \,\frac{1}{2} \int\limits_{\{\rho=t\}} 
\delta^{ij} \big( \pa_i g_{jk} - \pa_k g_{ij} \big) \nu^k \,d\sigma + O(t^{1-2\beta}) \,.
\end{align*}
The first summand in the right hand side is nonpositive by the {\em Euclidean Willmore inequality} (see~\cite{Will}), the second summand vanishes by the divergence theorem and it is well known (see~\cite[Proposition~4.1]{Bartnik}) that the third summand tends to $8 \pi m_{\rm ADM}$, as $t \to + \infty$. The starting statement then follows, as $\beta>1/2$.
In conclusion
$$
F=F_{11}(t)+F_{12}(t)+F_{13}(t)+F_{2}(t)\,,
$$
where
\begin{equation}
F_{11}(t)=-2\pi\mathcal{C}\,,\quad\,\,\lim_{t\to+\infty}F_{12}(t)\leq 8\pi m_{\mathrm{ADM}}\,,\quad \,\,\lim_{t\to+\infty}F_{13}(t)=0\,,\quad\,\, \lim_{t\to+\infty}F_{2}(t)=-6\pi\mathcal{C}\,,
\end{equation}
hence formula~\eqref{feq42} is proved.
\end{proof}

We are now in the position to state and prove the main result of this section.

\begin{theorem}\label{teo1 lower bound differenza tra il rapporto massa cap e una normalizzazione di Willmore energy}
Let $(M,g)$ be a $3$--dimensional, complete, one--ended asymptotically flat manifold with compact, connected boundary and with nonnegative scalar curvature.
Assume that $H_{2}(M,\partial M;\Z)=0$.
Let $u\in C^{\infty}(M)$ be the solution of Dirichlet problem~\eqref{f1prelc4} and let $\mathcal{C}>0$ be the boundary capacity of $\partial M$ in $(M,g)$, given by formula~\eqref{boundcapintermofu}. Then,
\begin{equation}\label{lower bound differenza tra il rapporto massa cap e una normalizzazione di Willmore energy}
\frac{m_{\mathrm{ADM}}}{\mathcal{C}}\,\geq\,\frac{5}{4}\,+\,\frac{1}{64\pi}\!\int\limits_{\partial M}\!\mathrm{H}^{2}\,d\sigma\,-\,\frac{1}{4\pi}\!\int\limits_{\partial M}\!\bigg(\vert \nabla u \vert\,+\,\frac{\mathrm{H}}{4}\bigg)^{2}\, d\sigma\,,
\end{equation}
with equality if and only if $(M,g)$ is isometric to a (exterior spatial) Schwarzschild manifold of mass $m>0$, see formula~\eqref{feq63bisbis}.
\end{theorem}
\begin{proof}
The assumption $H_{2}(M,\partial M;\Z)=0$ implies that all regular level sets of $u$ are connected, as said at the end of Section~\ref{settmasscapacity}, therefore, the function $F:[\mathcal{C}/2, + \infty)\to \R$, defined by formula~\eqref{feq40}, is nondecreasing on the set $\mathcal{T}$, given by equality~\eqref{defTcap}, by Proposition~\ref{propomonGcap}.
Hence, comparing the limit of the function $F(t)$ as $t \to+\infty$  with its value at $t = \mathcal{C}/2$, from Lemma~\ref{Fagliestremi} with the help of a small manipulation we obtain  
\begin{equation}\label{feq72}
\frac{m_{\mathrm{ADM}}}{\mathcal{C}}-1\geq \frac{1}{4}-\frac{1}{4\pi}\,\Bigg[\,\,\int\limits_{\partial M}\!\!\vert \nabla u \vert^{2}\, d\sigma+\frac{1}{2}\int\limits_{\partial M}\!\!\vert \nabla u \vert\,\mathrm{H}\,d\sigma\Bigg]\,.
\end{equation}
By adding and subtracting the term $(1/16)\!\int_{\partial M}\!\mathrm{H}^{2}\,d\sigma$ in the square brackets, one gets inequality~\eqref{lower bound differenza tra il rapporto massa cap e una normalizzazione di Willmore energy}.
Finally, the equivalence, i.e. the last part of the statement, follows from Proposition~\ref{Rigstat1}.
\end{proof}

\section{Second monotonicity formula}\label{SecMonforGandRigstat2}

The main result of the previous section provides an improvement of  Bray sharp mass--capacity inequality when the term on the right hand side of inequality~\eqref{lower bound differenza tra il rapporto massa cap e una normalizzazione di Willmore energy} is greater or equal to $1$, or equivalently, the right term of inequality~\eqref{feq72} is nonnegative.
In order to prove the latter, we find a second monotonicity formula, which, however, does not hold in the same setting of Theorem~\ref{teo1 lower bound differenza tra il rapporto massa cap e una normalizzazione di Willmore energy}, indeed, an assumption involving the mean curvature of the boundary is required. A key point of the followed strategy is to obtain, starting from our first monotonicity formula, a certain inequality along the regular values of $u$, therefore, also in this case, the main difficulty consists in ensuring that this inequality holds in presence of the critical
values of $u$, but this time the problem is overcome by means of a regularity result.
This regularity result and the proof of our second monotonicity formula are contained in the following proposition, while in the next we analyze the case of when the latter is constant.

\begin{proposition}\label{propG}
Let $(M,g)$ be a $3$--dimensional, complete, one--ended asymptotically flat manifold with compact, connected boundary and with nonnegative scalar curvature.
Let $u\in C^{\infty}(M)$ be the solution of Dirichlet problem~\eqref{f1prelc4} and let $\mathcal{C}>0$ be the boundary capacity of $\partial M$ in $(M,g)$, given by formula~\eqref{boundcapintermofu}.
Consider the function $G:[\mathcal{C}/2, + \infty)\to \R$ defined as
\begin{equation}\label{feq37}
G(t)= -\,\frac{\pi\mathcal{C}^{2}}{t}+\frac{t}{4} \,\left(1+\frac{\mathcal{C}}{2t}\right)^{\!4} \int\limits_{\Sigma_{t}}\vert \nabla u \vert^{2}\, d\sigma\,,
\end{equation}
where $\Sigma_{t}$ is the level set of $u$, given by
$$
\Sigma_{t}:=\Big\{u=\Big(1-\frac{\mathcal{C}}{2t}\Big)/\Big(1+\frac{\mathcal{C}}{2t}\Big)\Big\}\,,
$$ 
and $\sigma$ is the $2$--dimensional Hausdorff measure of $(M,g)$.
Then, the function $G$ satisfies
\begin{align}
G\left(\mathcal{C}/2\right)&=\,-\,2\mathcal{C}\Bigg[\pi-\int\limits_{\partial M}\vert \nabla u \vert^{2}\, d\sigma\Bigg]\,,\label{feq38}\\ 
\lim_{t\to+\infty}G(t)&=0\,.\label{feq39}
\end{align}
It is continuously differentiable on the set $\mathcal{T}$, given by identity~\eqref{defTcap}, with
\begin{equation}\label{G'}
G'(t)=\frac{\pi\mathcal{C}^{2}}{t^{2}} +\frac{1}{4}\,\left(1+\frac{\mathcal{C}}{2t}\right)^{\!3}\left(1-\frac{3\mathcal{C}}{2t}\right)\int\limits_{\Sigma_{t}}\vert \nabla u \vert^{2}\, d\sigma-\frac{\mathcal{C}}{4t}\left(1+\frac{\mathcal{C}}{2t}\right)^{\!2}\!\int\limits_{\Sigma_{t}}\vert \nabla u \vert\,\mathrm{H}\, d\sigma\,
\end{equation}
for every $t\in\mathcal{T}$,
and admits a locally absolutely continuous representative in $\left[\mathcal{C}/2,+\infty\right)$, (coinciding with it on $\mathcal{T}$). 
Finally, if all the regular level sets of $u$ are connected and there exists $$\alpha\in\big(\!-(2\mathcal{C})^{-1},\,(2\mathcal{C})^{-1}\big]$$ such that 
\begin{equation}\label{Assonmeancurvatureonboundary}
\,\,\,\,\,\qquad \HHH\leq \alpha\big(1-4\mathcal{C}\vert  \nabla u \vert\,\big)\quad\text{on}\,\,\partial M\,,
 \end{equation}
then $G$ is nondecreasing on the set $\mathcal{T}$.
\end{proposition}

Notice that the function $G$ is well defined, as the integrand function is bounded on each level set of $u$ and each level set of $u$ has finite $\sigma$--measure.

\begin{proof}
The function $G$ is easily seen to satisfy equality~\eqref{feq38}. 
Concerning limit~\eqref{feq39}, it is convenient to define the function $G_{1}$ as the second summand in definition~\eqref{feq37} of the function $G$, and to rewrite it as
$$
G_{1}(t)=\frac{\mathcal{C}}{4} \,\left(1+\frac{\mathcal{C}}{2t}\right)^{\!3} \!\!\!\!\int\limits_{\left\{u=f(t)\right\}}\!\!\!\!\frac{\vert \nabla u \vert}{1-u}\,\,\vert \nabla u \vert\, d\sigma\,,
$$
where $f:[\mathcal{C}/2, + \infty)\to[0,1)$ is the diffeomorphism defined by equality~\eqref{cap4f}.
By formulas~\eqref{hexpu} and~\eqref{feq48}, we have 
$$
\lim_{\vert x\vert\to + \infty}\frac{\vert \nabla u \vert}{1-u}= 0 \, ,
$$
therefore, $G_{1}(t)$ tends to zero for $t\to+\infty$, by applying an argument similar to the one leading to limit~\eqref{feq56}, thus, limit~\eqref{feq39} follows.\\
In absence of critical points, the function $G$ is everywhere continuously differentiable in its interval of definition, with first derivative given by
\begin{align*}
G'(t)\,=\,\frac{\pi\mathcal{C}^{2}}{t^{2}} \,+\,\frac{1}{4}\,\left(1+\frac{\mathcal{C}}{2t}\right)^{\!3}\left(1-\frac{3\mathcal{C}}{2t}\right)\int\limits_{\Sigma_{t}}\vert \nabla u \vert^{2}\, d\sigma\,-\,\frac{\mathcal{C}}{4t}\left(1+\frac{\mathcal{C}}{2t}\right)^{\!2}\!\int\limits_{\Sigma_{t}} \vert \nabla u \vert\,\mathrm{H}\, d\sigma\,,
\end{align*}
keeping into account formula~\eqref{cap4geq18}.\\
In presence of critical points, $G$ is continuously differentiable only on the set $\mathcal{T}$, with first derivative given as above.
In order to obtain the existence of its locally absolutely continuous representative, let us then show that $G_{1}\in W^{1,1}_{loc}(\mathcal{C}/2,+\infty)$.
Notice that $G_{1}\in L^{1}_{loc}(\mathcal{C}/2,+\infty)$, for instance by the coarea formula. 
Let $\chi\in C_{c}^{\infty}\big((\mathcal{C}/2,+\infty)\big)$. Then, one has
\begin{align}
\int\limits_{\mathcal{C}/2}^{+\infty}\!\chi'(t)\, G_{1}(t)dt&=\int\limits_{\mathcal{C}/2}^{+\infty}dt\,\bigg[\,\chi'(t)\,\frac{t}{4} \,\left(1+\frac{\mathcal{C}}{2t}\right)^{\!4} \!\!\!\int\limits_{\left\{u=f(t)\right\}}\!\!\!\!\vert \nabla u \vert^{2}\, d\sigma\bigg]\\
&=\int\limits_{0}^{1} ds \!\!\!\int\limits_{\left\{u=s\right\}}\!\!\!\!\chi'\!\left(\frac{\mathcal{C}}{2}\,\frac{1+u}{1-u}\right)\,\frac{2\mathcal{C}^{2}}{(1-u^{2})^{3}}\,\vert \nabla u \vert^{2}\, d\sigma\\
&=\int\limits_{M}\chi'\!\left(\frac{\mathcal{C}}{2}\,\frac{1+u}{1-u}\right)\,\frac{2\mathcal{C}^{2}}{(1-u^{2})^{3}}\,\vert \nabla u \vert^{3}\, d\mu\\
&=\lim\limits_{k\to+\infty}\int\limits_{M}\chi'\!\left(\frac{\mathcal{C}}{2}\,\frac{1+u}{1-u}\right)\,\frac{2\mathcal{C}^{2}}{(1-u^{2})^{3}}\,\eta_{k}(\vert \nabla u \vert^{2})\,\vert \nabla u \vert^{3}\, d\mu\\
&=\lim\limits_{k\to+\infty}\int\limits_{M}g\!\left(\nabla\!\left[\chi\!\left(\frac{\mathcal{C}}{2}\,\frac{1+u}{1-u}\right)\right] ,\frac{2\mathcal{C}\vert \nabla u \vert}{(1-u)(1+u)^{3}}\,\eta_{k}(\vert \nabla u \vert^{2}) \na u\!\right)\!d\mu\\
&=-\lim\limits_{k\to+\infty}\int\limits_{M}\chi\!\left(\frac{\mathcal{C}}{2}\,\frac{1+u}{1-u}\right)\mathrm{div}\!\left(\!\frac{2\mathcal{C}\vert \nabla u \vert}{(1-u)(1+u)^{3}}\,\eta_{k}(\vert \nabla u \vert^{2}) \na u\!\right)\!d\mu\,.\quad\label{cap4geq19}
\end{align}
Here, the third equality is a consequence of the coarea formula, the fourth follows by the dominate convergence theorem, since
the sequence of the functions $\eta_{k}(\vert \nabla u \vert^{2})$ converges pointwise on $M$ to the function $\mathbb{I}_{M\setminus\mathrm{Crit}(u)}$ and 
\begin{align*}
\bigg\vert\chi'\!\left(\frac{\mathcal{C}}{2}\,\frac{1+u}{1-u}\right)\,\frac{2\mathcal{C}^{2}}{(1-u^{2})^{3}}\,\eta_{k}(\vert \nabla u \vert^{2})\,\vert \nabla u \vert^{3} \bigg\vert\leq \frac{2\mathcal{C}^{2}\!\parallel\! \chi'\!\parallel_{L^\infty(\mathcal{C}/2,+\infty)}\,}{(1-u^{2})^{3}}\,\vert \nabla u \vert^{3}\in L^{1}_{loc}(M)\,,
\end{align*}
finally, the last equality is obtained by the properties of the divergence operator combined with the divergence theorem applied to the vector field
$$
\chi\Big(\frac{\mathcal{C}}{2}\,\frac{1+u}{1-u}\Big)\frac{2\mathcal{C}\vert \nabla u \vert}{(1-u)(1+u)^{3}}\,\eta_{k}(\vert \nabla u \vert^{2})\na u
$$
in the domain $E_{a}^{b}:=\left\{f(a)< u<f(b)\right\}$, for $a,b\in\mathcal{T}$ such that $\mathrm{supp}\chi\subseteq(a,b)$.
We observe that
\begin{align}
\int\limits_{M}&\chi\!\left(\frac{\mathcal{C}}{2}\,\frac{1+u}{1-u}\right)\mathrm{div}\!\left(\!\frac{2\mathcal{C}\vert \nabla u \vert}{(1-u)(1+u)^{3}}\,\eta_{k}(\vert \nabla u \vert^{2}) \na u\!\right)\!d\mu\\
&=\!\int\limits_{E_{a}^{b}}\!\chi\!\left(\frac{\mathcal{C}}{2}\,\frac{1+u}{1-u}\right)\eta_{k}(\vert \nabla u \vert^{2})\,\frac{\mathcal{C}\vert \nabla u \vert}{(1-u)^{2}}\left[\frac{4(2u-1)}{(1+u)^{4}}\,\vert \nabla u \vert^{2}+\frac{2(1-u)}{(1+u)^{3}} \,\frac{g(\na \vert \nabla u \vert, \na u)}{ \vert \nabla u \vert}\right]d\mu\\
&\quad+\!\int\limits_{E_{a}^{b}}\!\chi\!\left(\frac{\mathcal{C}}{2}\,\frac{1+u}{1-u}\right)\eta'_{k}(\vert \nabla u \vert^{2})\, \frac{4\mathcal{C}\vert \nabla u \vert^{2}}{(1-u)(1+u)^{3}}\,\, g(\na \vert \nabla u \vert, \na u)\,d\mu\,,\quad\label{feq52}
\end{align}
where
\begin{equation}
\bigg\vert\int\limits_{E_{a}^{b}}\!\chi\!\left(\frac{\mathcal{C}}{2}\,\frac{1+u}{1-u}\right)\eta'_{k}(\vert \nabla u \vert^{2})\, \frac{4\mathcal{C}\vert \nabla u \vert^{2}}{(1-u)(1+u)^{3}}\,\, g(\na \vert \nabla u \vert, \na u)\,d\mu\,\bigg\vert
\leq \frac{C}{\sqrt{2k}}\longrightarrow0\,\,\,\text{for}\,\,\,k\to+\infty\,,\label{cap4geq20}
\end{equation}
as
\begin{equation}
\Big \vert \, \eta'_{k}(\vert \nabla u \vert^{2})\,\vert \nabla u \vert^{2} \, g(\na \vert \nabla u \vert, \na u)\,\Big\vert\leq\eta'_{k}(\vert \nabla u \vert^{2})\,\vert \nabla u \vert^{3}\,\vert  \nabla du \vert\,\mathbb{I}_{\{\frac{1}{2k}\leq\vert \nabla u \vert^{2}\leq\frac{3}{2k} \}}\leq \frac{3^{3/2}}{\sqrt{2k}}\,\vert  \nabla du \vert\,,
\end{equation}
and
\begin{equation}
\lim\limits_{k\to+\infty}\int\limits_{E_{a}^{b}}\!\chi\!\left(\frac{\mathcal{C}}{2}\,\frac{1+u}{1-u}\right)\eta_{k}(\vert \nabla u \vert^{2})\,\frac{\mathcal{C}\vert \nabla u \vert}{(1-u)^{2}}\,Q\,d\mu=
\int\limits_{M}\!\chi\!\left(\frac{\mathcal{C}}{2}\,\frac{1+u}{1-u}\right)\frac{\mathcal{C}\vert \nabla u \vert}{(1-u)^{2}}\,Q\,d\mu\,,\quad\quad
\label{cap4geq21}
\end{equation}
having set
$$
Q:=\frac{4(2u-1)}{(1+u)^{4}}\,\vert \nabla u \vert^{2}+\frac{2(1-u)}{(1+u)^{3}} \,\frac{g(\na \vert \nabla u \vert, \na u)}{ \vert \nabla u \vert}\,,
$$
by the dominate convergence theorem, indeed, $\eta_{k}(\vert \nabla u \vert^{2})\to \mathbb{I}_{M\setminus\mathrm{Crit}(u)}$ pointwise on $M$ for $k\to +\infty$ and
\begin{equation}
\vert Q\vert\leq \frac{4 \vert 2u-1\vert}{(1+u)^{4}}\,\vert \nabla u \vert^{2}+\frac{2(1-u)}{(1+u)^{3}}\, \vert  \nabla du \vert \,\in L^{1}_{loc}(M)\,,\label{feq53}
\end{equation}
\begin{equation}
\left\vert\chi\!\left(\frac{\mathcal{C}}{2}\,\frac{1+u}{1-u}\right) \eta_{k}(\vert \nabla u \vert^{2})\,\frac{\mathcal{C}\vert \nabla u \vert}{(1-u)^{2}}\,Q\right\vert\leq\frac{\mathcal{C}\|\chi\|_{L^{\infty}(\mathcal{C}/2,+\infty)}}{(1-u)^{2}}\,\vert \nabla u \vert\,\vert Q\vert\,\in L^{1}_{loc}(M)\,.
\end{equation}
Notice that above, the function $Q$ is well defined in $M\setminus\mathrm{Crit}(u)$, hence, these last inequalities hold $\mu$--a.e., as $\mu(\mathrm{Crit}(u))=0$.\\
Then, from formula~\eqref{cap4geq19}, by virtue of equality~\eqref{feq52} together with limits~\eqref{cap4geq20} and~\eqref{cap4geq21}, it follows 
\begin{align}
&\int\limits_{\mathcal{C}/2}^{+\infty}\!\chi'(t)\, G_{1}(t)\,dt\\
&\quad=-\int\limits_{M}\!\chi\!\left(\frac{\mathcal{C}}{2}\,\frac{1+u}{1-u}\right)\frac{\mathcal{C}\vert \nabla u \vert}{(1-u)^{2}}\!\left[\frac{4(2u-1)}{(1+u)^{4}}\,\vert \nabla u \vert^{2}+\frac{2(1-u)}{(1+u)^{3}} \,\frac{g(\na \vert \nabla u \vert, \na u)}{ \vert \nabla u \vert}\right]\,d\mu\\
&\quad=-\!\!\!\int\limits_{\mathcal{C}/2}^{+\infty}\!dt\,\chi(t)\bigg[\frac{1}{4}\!\left(1+\frac{\mathcal{C}}{2t}\right)^{\!3}\!\!\left(1-\frac{3\mathcal{C}}{2t}\right)\!\int\limits_{\Sigma_{t} }\vert \nabla u \vert^{2}\, d\sigma-\frac{\mathcal{C}}{4t}\!\left(1+\frac{\mathcal{C}}{2t}\right)^{\!2}\!\int\limits_{\Sigma_{t} }\vert \nabla u \vert\,\mathrm{H}\, d\sigma\bigg]\,,
\end{align}
where the second equality is obtained by the coarea formula.
In this way, we conclude that $G_{1}$ has a weak derivative in the open interval $(\mathcal{C}/2,+\infty)$, which is in $L^{1}_{loc}(\mathcal{C}/2,+\infty)$, as a consequence of the fact that each summand of $\mathcal{C}\,Q/(1-u)^{2}$ is in $L^{1}_{loc}(M)$ and of the coarea formula.
Then, $G_{1}\in W^{1,1}_{loc}(\mathcal{C}/2,+\infty)$. 
Consequently, $G$ admits a locally absolutely continuous representative in $\left[\mathcal{C}/2,+\infty\right)$, as $G(t)= - \pi\mathcal{C}^{2}/t+G_{1}(t)$, coinciding with $G$ on $\mathcal{T}$.\\ 
Being the function $G$ continuously differentiable on $\mathcal{T}$, with first derivative given by formula~\eqref{G'}, it follows easily the equality
\begin{equation}\label{eqf50}
F(t)=\,\frac{4t^{3}}{\mathcal{C}^{2}}\, G'(t)
\end{equation}
for every $t\in \mathcal{T}$.
We set
\begin{equation}
\qquad\qquad\qquad\mathcal{A}:=2\mathcal{C}\Bigg[\pi-\int\limits_{\partial M}\vert \nabla u \vert^{2}\, d\sigma\Bigg]\qquad\text{and}\qquad\mathcal{B}:=\mathcal{C}\Bigg[2\pi-2\!\int\limits_{\partial M}\!\!\vert \nabla u \vert^{2}\, d\sigma-\!\int\limits_{\partial M}\!\!\vert \nabla u \vert\,\mathrm{H}\,d\sigma\Bigg]\,,
\end{equation}
therefore,
\begin{equation}\label{feq64}
G\left(\mathcal{C}/2\right)=-\mathcal{A}\qquad\text{and}\qquad F\left(\mathcal{C}/2\right)=\mathcal{B}\,,
\end{equation}
by formulas~\eqref{feq38} and~\eqref{feq41}, respectively, and 
$$\mathcal{B}=\mathcal{A}-\mathcal{C}\!\int\limits_{\partial M}\!\!\vert \nabla u \vert\,\mathrm{H}\,d\sigma\,.$$
Then, as a result of this last equality,
\begin{equation}\label{feq63}
\mathcal{B}\geq\big(1-2\mathcal{C}\alpha\big)\mathcal{A}\,,
\end{equation}
since assumption~\eqref{Assonmeancurvatureonboundary} implies 
\begin{equation}
\int\limits_{\partial M}\!\!\vert \nabla u \vert\,\mathrm{H}\,d\sigma\leq 2\alpha \mathcal{A}
\end{equation}
by identity~\eqref{feq44}.
The monotonicity of $F$, proved in Proposition~\ref{propomonGcap} under the assumption that all regular level sets of $u$ are connected, implies 
\begin{equation}
\frac{4t^{3}}{\mathcal{C}^{2}}\,G'(t)-\mathcal{B}=F(t)-F\left(\mathcal{C}/2\right)\geq 0
\end{equation}
for every $t\in\mathcal{T}$, hence, 
\begin{equation}\label{feq65}
G'(t)\geq\, \frac{\mathcal{C}^{2}}{4t^{3}}\,\mathcal{B}
\end{equation}
for all $t\in\mathcal{T}$.
Notice that this inequality is true a.e. in $[\mathcal{C}/2,+\infty)$, as $\mathcal{T}$ differs from $[\mathcal{C}/2,+\infty)$ only for a negligible set, by Sard theorem.
Then, integrating between $\mathcal{C}/2$ and $t\in\mathcal{T}$ and since $G$ admits a locally absolutely continuous representative in $\left[\mathcal{C}/2,+\infty\right)$ coinciding with it on $\mathcal{T}$, it follows
\begin{equation}
G(t)-G\left(\mathcal{C}/2\right)\geq -\frac{\mathcal{C}^{2}\mathcal{B}}{8t^{2}}\,+\,\frac{\mathcal{B}}{2}
\end{equation}
for every $t\in\mathcal{T}$, consequently, 
\begin{equation}
G(t)\geq-\frac{\mathcal{C}^{2}\mathcal{B}}{8t^{2}}\,-\,\frac{(1+2\mathcal{C}\alpha)\mathcal{A}}{2}\label{cap4geq22}
\end{equation}
for all $t\in\mathcal{T}$, by formula~\eqref{feq64} and inequality~\eqref{feq63}.
As explained before Lemma~\ref{Fagliestremi}, there exists $t_{0}\in \left[\mathcal{C}/2,+\infty\right)$ such that $[t_{0},+\infty)\subseteq \mathcal{T}$, therefore, passing in inequality~\eqref{cap4geq22} to the limit for $t\to +\infty$, we get $(1+2\mathcal{C}\alpha)\mathcal{A}\geq0$ from limit~\eqref{feq39}. 
Being $\alpha\in\big(\!-(2\mathcal{C})^{-1},\,(2\mathcal{C})^{-1}\big]$, then $\mathcal{A}\geq0$, from which it follows $\mathcal{B}\geq 0$, by inequality~\eqref{feq63}.
Thus, $G'(t)\geq0$ for every $t\in\mathcal{T}$, by inequality~\eqref{feq65}, and this implies that $G$ is nondecreasing on the set $\mathcal{T}$ (since $\mathcal{T}$ differs from $[\mathcal{C}/2,+\infty)$ for a negligible set and $G$ admits a locally absolutely continuous representative in $\left[\mathcal{C}/2,+\infty\right)$, coinciding with it on $\mathcal{T}$).
\end{proof}

\begin{proposition}[Rigidity -- II]\label{Rigstat2}
Let $(M,g)$ be a $3$--dimensional, complete, one--ended asymptotically flat manifold with compact, connected boundary and with nonnegative scalar curvature.
Let $u\in C^{\infty}(M)$ be the solution of Dirichlet problem~\eqref{f1prelc4} and let $\mathcal{C}>0$ be the boundary capacity of $\partial M$ in $(M,g)$, given by formula~\eqref{boundcapintermofu}. Consider the function $G:[\mathcal{C}/2, + \infty)\to \R$ defined by equality~\eqref{feq37}.
Then, $G$ is constant on the set $\mathcal{T}$, given by equality~\eqref{defTcap}, if and only if $(M,g)$ is isometric to the (exterior spatial) Schwarzschild manifold $(M_{\mathrm{Sch}(m)},\,g_{\mathrm{Sch}(m)})$ of mass $m>0$, see formula~\eqref{feq63bisbis}.
\end{proposition}
\begin{proof}
If $(M,g)$ is the (exterior spatial) Schwarzschild manifold with mass $m>0$, then, by the equalities~\eqref{heq1} along with the observation $m=\mathcal{C}$, one obtains directly that $G\equiv 0$ in $[m/2, + \infty)$.
Now, we assume that $G$ is constant on the set $\mathcal{T}$, then, as in the proof of Proposition~\ref{Rigstat1}, there exists a maximal time $T\in(\mathcal{C}/2, + \infty]$ such that $\na u \neq 0$ in $M_{T}:=\big\{0\leq u<(1-\frac{\mathcal{C}}{2T})/(1+\frac{\mathcal{C}}{2T})\big\}$, since $\mathcal{T}\supseteq[a, b)$ with $a=\mathcal{C}/2$. 
Hence, in $\left[\frac{\mathcal{C}}{2},T\right)$ the function $G$ is of class $C^{2}$, with $G'(t)$ given by formula~\eqref{G'} and at same time with $G'(t)=0$, while $F$ is of class $C^{1}$, with $F'$ given by formula~\eqref{feq54} and at the same time $F'(t)=0$, as $F\equiv0$ due to equality~\eqref{eqf50}.
Consequently, arguing as in the proof of Proposition~\ref{Rigstat1}, one obtains that $T= +\infty$ and $|\na u| \neq 0$ everywhere.
Then, the conclusion follows from Proposition~\ref{Rigstat1}.
\end{proof}

We are now ready to state and to prove the two main results of this section.

\begin{theorem}\label{teocapacityandareaboundarybis}
Let $(M,g)$ be a $3$--dimensional, complete, one--ended asymptotically flat manifold with compact, connected boundary and with nonnegative scalar curvature.
Let $u\in C^{\infty}(M)$ be the solution of Dirichlet problem~\eqref{f1prelc4} and let $\mathcal{C}>0$ be the boundary capacity of $\partial M$ in $(M,g)$, given by formula~\eqref{boundcapintermofu}.
Assume that $H_2(M,\partial M;\Z) = 0$ and there exists $\alpha\in\big(\!-(2\mathcal{C})^{-1},\,(2\mathcal{C})^{-1}\big]$ such that 
$ \HHH\leq \alpha\big(1-4\mathcal{C}\vert  \nabla u \vert\,\big)$ on $\partial M$.
Then,
\begin{align}\label{feq62bis}
4\pi\,t^{2}\,\bigg(\!1+\frac{\mathcal{C}}{2t}\bigg)^{\!4}\leq\mathrm{Area}\left(\Biggl\{u=\frac{1-\frac{\mathcal{C}}{2t}}{1+\frac{\mathcal{C}}{2t}}\Biggr\}\right)\,\,
\end{align}
for every $t\in \mathcal{T}$, where the set $\mathcal{T}$ is given by equality~\eqref{defTcap}. Thus,
\begin{equation}\label{capacityandareaboundarybis}
\mathcal{C}\leq\sqrt{\frac{\mathrm{Area}(\partial M)}{16\pi}} \,,
\end{equation}
with equality if and only if $(M,g)$ is isometric to a (exterior spatial) Schwarzschild manifold of mass $m>0$, see formula~\eqref{feq63bisbis}.
\end{theorem}

This theorem provides a sharp comparison between the area of the level sets of the function $u$ and of the analogue of $u$ in the Schwarzschild manifold of mass $\mathcal{C}$. 

\begin{proof}
The assumption $H_{2}(M,\partial M;\Z)=0$ implies that all regular level sets of $u$ are connected, as already explained at the end of Section~\ref{settmasscapacity}. Then, the function $G$, defined by formula~\eqref{feq37}, is nondecreasing on the set $\mathcal{T}$, by Propositions~\ref{propG}. Thus, for every $t\in \mathcal{T}$
\begin{equation}
G(\mathcal{C}/2) \leq \,G(t)\leq \lim_{t\to+\infty}G(t)=0\,,\label{feq60bis}
\end{equation}
by limit~\eqref{feq39}.
The last inequality and definition~\eqref{feq37} of $G$ imply 
$$
\int\limits_{\Sigma_{t}} \vert \nabla u \vert^{2}\, d\sigma\leq \frac{4\pi\mathcal{C}^{2}}{t^{2}}\,\bigg(\!1+\frac{\mathcal{C}}{2t}\bigg)^{\!-4}\,,
$$
where $\Sigma_{t}$ is the level set of $u$ given by formula~\eqref{Sigmat}.
As a consequence,
\begin{equation}\label{feq68}
4\pi\mathcal{C}=\int\limits_{\Sigma_{t}} \vert \nabla u \vert\, d\sigma\leq\Bigg[\,\int\limits_{\Sigma_{t}} \vert \nabla u \vert^{2}\, d\sigma\Bigg]^{1/2}\!\!\left[\mathrm{Area}(\Sigma_{t})\right]^{1/2}
\leq (4\pi)^{1/2}\,\,\frac{\mathcal{C}}{t}\,\left(1+\frac{\mathcal{C}}{2t}\right)^{\!-2}\!\!\left[\mathrm{Area}(\Sigma_{t})\right]^{1/2}\,,
\end{equation}
where the equality is known from formula~\eqref{feq44} and the first inequality follows from H\"{o}lder inequality.
Then,
\begin{align*}
\mathrm{Area}(\Sigma_{t})\geq 4\pi\,t^{2}\,\bigg(\!1+\frac{\mathcal{C}}{2t}\bigg)^{\!4}\,
\end{align*}
for every $t\in \mathcal{T}$. In particular, for $t=\mathcal{C}/2\in \mathcal{T}$, one has
\begin{align*}
\mathrm{Area}(\partial M)&\geq 16\,\pi\,\mathcal{C}^{2}\,,
\end{align*}
from which it follows
$$
\mathcal{C}\leq\sqrt{\frac{\mathrm{Area}(\partial M)}{16\pi}}\,.
$$
Finally, if we assume that the equality holds, then, for $t=\mathcal{C}/2$, the chain of inequalities~\eqref{feq68} is a chain of equalities, therefore, $G(\mathcal{C}/2)=0$. Thus, $G$ is constant on $\mathcal{T}$, by formula~\eqref{feq60bis}, as a consequence, $(M,g)$ is isometric to $(M_{\mathrm{Sch}(\mathcal{C})},\,g_{\mathrm{Sch}(\mathcal{C})})$, by Proposition~\ref{Rigstat2}. On other side, in a (exterior spatial) Schwarzschild manifold with mass $m>0$, the equality in formula~\eqref{capacityandareaboundarybis} can be checked directly.
\end{proof}

\begin{theorem}\label{teo2 lower bound differenza tra il rapporto massa cap e una normalizzazione di Willmore energybis}
Let $(M,g)$ be a $3$--dimensional, complete, one--ended asymptotically flat manifold with compact, connected boundary and with nonnegative scalar curvature.
Let $u\in C^{\infty}(M)$ be the solution of Dirichlet problem~\eqref{f1prelc4} and let $\mathcal{C}>0$ be the boundary capacity of $\partial M$ in $(M,g)$, given by formula~\eqref{boundcapintermofu}.
Assume that $H_2(M,\partial M;\Z) = 0$ and there exists $\alpha\in\big(\!-(2\mathcal{C})^{-1},\,(2\mathcal{C})^{-1}\big]$ such that 
$ \HHH\leq \alpha\big(1-4\mathcal{C}\vert  \nabla u \vert\,\big)$ on $\partial M$.
Then,
\begin{equation}\label{lower bound differenza tra il rapporto massa cap e una normalizzazione di Willmore energybis}
\frac{m_{\mathrm{ADM}}}{\mathcal{C}}\,\geq\,\frac{5}{4}\,+\,\frac{1}{64\pi}\!\int\limits_{\partial M}\!\mathrm{H}^{2}\,d\sigma\,-\,\frac{1}{4\pi}\!\int\limits_{\partial M}\!\bigg(\vert \nabla u \vert\,+\,\frac{\mathrm{H}}{4}\bigg)^{2}\, d\sigma\,\geq 1,
\end{equation}
with the equality in the first inequality if and only if $(M,g)$ is isometric to a (exterior spatial) Schwarzschild manifold of mass $m>0$, see formula~\eqref{feq63bisbis}.
\end{theorem}
\begin{proof}
Notice that everything is already known by Theorem~\ref{teo1 lower bound differenza tra il rapporto massa cap e una normalizzazione di Willmore energy}, except the fact that the central term of inequality~\eqref{lower bound differenza tra il rapporto massa cap e una normalizzazione di Willmore energybis} is greater or equal to $1$, but it follows from the proof of Proposition~\ref{propG}, since this term is equal to $1+\mathcal{B}/(8\pi\mathcal{C})$.
\end{proof}

Under the hypothesis of the previous theorem with the difference of assuming the existence of $\alpha\in\big(\!-(2\mathcal{C})^{-1},\,(2\mathcal{C})^{-1}\big)$ such that $\HHH\leq \alpha\big(1-4\mathcal{C}\vert  \nabla u \vert\,\big)$ on $\partial M$, the term to the right hand side of inequality~\eqref{feq72} is nonnegative, as it is equal to $\mathcal{B}/(8\pi\mathcal{C})$, and it is zero (or, equivalently, the central term of inequality~\eqref{lower bound differenza tra il rapporto massa cap e una normalizzazione di Willmore energybis} is equal to $1$)  if and only if $(M,g)$ is isometric to a (exterior spatial) Schwarzschild manifold of mass $m>0$, indeed, the equality holds if and only if $\mathcal{B}=0$, which is equivalent to have $\mathcal{A}=0$, as $\mathcal{B}\geq\big(1-2\mathcal{C}\alpha\big)\mathcal{A}\geq 0$ (see the proof of Proposition~\ref{propG}).
This in turn is equivalent to say that $G$ is constant on the set $\mathcal{T}$, thus, the conclusion follows by Proposition~\ref{Rigstat2}.

\begin{remark}
If $\mathrm{H}\leq 0$, the condition of the existence of a real number $\alpha\in\big(\!-(2\mathcal{C})^{-1},\,(2\mathcal{C})^{-1}\big)$, such that 
$ \HHH\leq \alpha\big(1-4\mathcal{C}\vert  \nabla u \vert\,\big)$ on $\partial M$, is naturally satisfied, indeed, one can take $\alpha = 0$.
\end{remark}

As an immediate corollary of the above theorem, we have the following extension of the cases of validity of the mass--capacity inequality obtained by Bray in~\cite{bray1}.

\begin{cor}\label{genBray}
Let $(M,g)$ be a $3$--dimensional, complete, one--ended asymptotically flat manifold with compact, connected boundary and with nonnegative scalar curvature.
Let $u\in C^{\infty}(M)$ be the solution of Dirichlet problem~\eqref{f1prelc4} and let $\mathcal{C}>0$ be the boundary capacity of $\partial M$ in $(M,g)$, given by formula~\eqref{boundcapintermofu}.
Assume that $H_2(M,\partial M;\Z) = 0$ and there exists $\alpha\in\big(\!-(2\mathcal{C})^{-1},\,(2\mathcal{C})^{-1}\big]$ such that 
$ \HHH\leq \alpha\big(1-4\mathcal{C}\vert  \nabla u \vert\,\big)$ on $\partial M$.
Then,
\begin{equation}\label{capacityandmass}
m_{\mathrm{ADM}}\,\geq\,\mathcal{C}\,,
\end{equation}
with the equality if and only if $(M,g)$ is isometric to a (exterior spatial) Schwarzschild manifold of mass $m>0$, see formula~\eqref{feq63bisbis}.
\end{cor}

Notice that all results contained in this section continue to be true if we replace the assumption of existence of $\alpha\in\big(\!-(2\mathcal{C})^{-1},\,(2\mathcal{C})^{-1}\big]$ (resp. $\alpha\in (-(2\mathcal{C})^{-1},\,(2\mathcal{C})^{-1})$) such that $ \HHH\leq \alpha\big(1-4\mathcal{C}\vert  \nabla u \vert\,\big)$ on $\partial M$, with the assumption of existence of $\alpha\in\big(\!-(2\mathcal{C})^{-1},\,(2\mathcal{C})^{-1}\big]$ (resp. $\alpha\in (-(2\mathcal{C})^{-1},\,(2\mathcal{C})^{-1})$) such that inequality~\eqref{feq63} holds.

\appendix

\section{Some topological remarks}\label{AppA}

In this appendix, we provides an alternative approach to prove that the regular level sets of the solution $u$ of Dirichlet problem~\eqref{f1prelc4}, in a $3$--dimensional, complete, one--ended asymptotically flat manifold $(M,g)$ with compact, connected boundary, are connected. 
With this aim, we show some topological results involving smooth manifolds with boundary and refer the reader to~\cite{Hatcher, bredon, massey2, lee3} for the basic ones.
For completeness, we also treat the case of smooth manifolds without boundary.
First let us start with the following lemma.

\begin{lemma}\label{compcase}
Let $N$ be a $3$--dimensional, compact, smooth manifold with or without connected boundary $\partial N$. Then the first Betti number of $N$, which is the rank of $H_{1}(N;\Z)$, is zero if and only of $H_{2}(N,\partial N;\Z)=0$.
\end{lemma}

\begin{proof}
We divide the discussion into six cases.
\begin{enumerate}[label=$\mathrm{(\arabic*)}$]
\item {\em $N$ is orientable and $\partial N=\emptyset$.}
\item {\em $N$ is nonorientable and $\partial N=\emptyset$.}
\item {\em $N$ is orientable and $\partial N\neq \emptyset$.}
\item {\em $N$ is nonorientable and $\partial N$ is a $2$--sphere.}
\item {\em $N$ is nonorientable and $\partial N$ is a connected sum of $k$ tori, $k\in \N^{+}$.}
\item {\em $N$ is nonorientable and $\partial N$ is a connected sum of $2k$ projective planes, $k\in \N^{+}$.}
\end{enumerate}
Note that the case where $N$ is nonorientable and $\partial N$ is a connected sum of $2k+1$ projective planes, with $k\in \N$, is not possible because $\chi (\partial N)=2\chi (N)$ is even, by using the odd dimension of $N$.
Here, $\chi (\partial N)$ and $\chi (N)$ are the Euler--Poincar\`e characteristic of $\partial N$ and $N$, respectively.\\
In case~$\mathrm{(1)}$, the equivalence is a consequence of the Poincar\`e duality theorem and the fact that $H_{2}(N;\Z)$ is a torsion--free group.
In case~$\mathrm{(2)}$, the torsion subgroup of $H_{2}(N;\Z)$ is $\Z_{2}$ and the first Betti number of $N$ is nonzero, indeed, $b_{1}(N)=1+b_{2}(N)$. This last equality is obtained by combining
$\chi(N)=0$ and $H_{3}(N;\Z)=0$, which hold true due to the odd dimension and the nonorientability of $N$, respectively. In case~$\mathrm{(3)}$, similarly to case~$\mathrm{(1)}$, the equivalence is a consequence of the Poincar\`e duality theorem and the fact that $H_{2}(N,\partial N;\Z)$ is a torsion--free group. Concerning case~$\mathrm{(4)}$, using the Mayer--Vietoris sequence for reduced homology and the assumption that $\partial N$ is a $2$-sphere, one obtains $2b_ 1 (N) =b_ 1 (\mathrm {D} (N)) $, where $\mathrm {D} (N)$ is the double of $N$. As a result, $b_{1}(N)>0$, since the manifold $\mathrm{D}(N)$ lays in case~$\mathrm{(2)}$. This implies $b_{2}(N)>0$ due to the equality $b_{2}(N)=b_{1}(N)$, which is gotten by exploiting in the identity $\chi (\partial N)=2\chi (N)$ the assumption that $\partial N$ is a $2$-sphere together with the equalities $b_{0}(N)=1$, $b_{3}(N)=0$. More precisely, $H_{0}(N;\Z)$ is $\Z$ and $H_{3}(N;\Z)$ is zero. Notice that these equalities hold in general when $N$ is connected and has boundary, respectively.
Suppose, by contradiction, that $H_{2}(N,\partial N;\Z)$ is zero. Then, one has the following exact short sequence
\begin{equation}
\big(H_{3}(N;\Z)=\!\big) 0\to H_3(N,\partial N;\Z)\to H_2(\partial N;\Z)\to H_2(N;\Z)\to 0 \big(\!=H_2(N,\partial N;\Z)\big)\,,
\end{equation}
hence $H_3(N,\partial N;\Z)$ is torsion--free and 
\begin{equation}\label{apAeq1}
1=b_{2}(\partial N)=b_{3}(N,\partial N)+b_{2}(N)\,,
\end{equation}
since $H_2(\partial N;\Z)$ is torsion--free and equal to $\Z$, respectively.
As a consequence of formula~\eqref{apAeq1} with the result $b_{2}(N)>0$, one has $b_{3}(N,\partial N)=0$, but then $H_3(N,\partial N;\Z)$ is zero because it is also torsion--free.
Thus, by the universal coefficient theorem for homology, $H_3(N,\partial N;\Z_{2})$ is zero, but this is no possible because $H_3(N,\partial N;\Z_{2})=\Z_{2}$.
In the remaining cases, one has immediately $b_{1}(N)>0$, by using in the identity $\chi (\partial N)=2\chi (N)$ the assumption that $\partial N$ is the connected sum of $k$ tori or $2k$ projective planes.
Concerning case~$\mathrm{(5)}$, by Mayer--Vietoris, one gets the following long exact sequence
$$\dots\to\big( H_{3}(\mathrm{D}(N);\Z)=\!\big)0\to H_{2}(\partial N;\Z) \to H_{2}(N;\Z)\oplus H_{2}(N;\Z)\to \dots\,,$$
hence $b_{2}(N)>0$, being $H_{2}(\partial N;\Z) =\Z$ in this case. Then, the conclusion that $H_2(N,\partial N;\Z)$ is not zero follows in the same way of case~$\mathrm{(4)}$.
Case~$\mathrm{(6)}$ is simpler than cases~$\mathrm{(4)}$ and~$\mathrm{(5)}$. Indeed, if by contradiction $H_2(N,\partial N;\Z)$ is zero, then one has the following short exact sequence
\begin{equation}
\big(H_{3}(N;\Z)=\!\big) 0\to H_3(N,\partial N;\Z)\to 0\big(\!=H_2(\partial N;\Z)\big)\,,
\end{equation}
which implies trivially that $H_3(N,\partial N;\Z)$ is zero, but this is not possible, as argued at the end of case~$\mathrm{(4)}$.
\end{proof}

We emphasize that if $N$ is a smooth manifold as considered in the preceding lemma and meets one of the previous equivalent conditions, then $N$ is always orientable. Moreover, if $\partial N\neq \emptyset$, then $\partial N$ is a $2$-sphere. 
This follows using the exact long sequence
\begin{equation}
\dots\to H_2(N;\Z)\to H_2(N,\partial N;\Z)\to H_{1}(\partial N;\Z)\to H_{1}(N;\Z)\to \dots\,,
\end{equation}
together with the characterization theorem of the closed, connected surfaces.

In the following lemma, we extend the previous result to the noncompact cases of our interest.

\begin{proposition}\label{noncompcase}
Let $M$ be a $3$--dimensional, noncompact, smooth manifold with or without connected, compact boundary $\partial M$.
Assume that there exists a compact subset $K\subseteq M$ such that $M\setminus K$ is diffeomorphic to $\R^{3}$ minus a closed ball.
Then, the first Betti number of $M$ is zero if and only if $H_{2}(M,\partial M;\Z)=0$.
\end{proposition}
\begin{proof}
We treat the cases $\partial M = \emptyset$ and $\partial M \neq \emptyset$ separately.
In the first case, i.e. $M$ is without boundary, as a consequence of the made assumption, firstly, $M$ is diffeomorphic to a $3$--dimensional, closed, smooth manifold $N$ minus its point $P$, and secondly, one has the following exact short sequence
\begin{equation}\label{apAeq2}
0\to H_{3}(N;\Z) \to \Z\to H_{2}(M;\Z)\to  H_{2}(N;\Z)\to 0\,,
\end{equation}
by Mayer--Vietoris. 
Thus, if $H_{2}(M;\Z)$ is zero, then $H_{2}(N;\Z)$ is zero too.
Consequently, we know by Lemma~\ref{compcase} that $b_{1}(N)=0$, which implies $b_{1}(M)=0$, since $H_{1}(M;\Z)$ and $H_{1}(N;\Z)$ are isomorphic groups.
This last statement is true because $M$ is diffeomorphic to $N\setminus\{P\}$.
Therefore, if $b_{1}(M)=0$, then $b_{1}(N)=0$ and again by Lemma~\ref{compcase}, one has $H_{2}(N;\Z)=0$.
Notice that $N$ is orientable, as  emphasized before, hence $H_{3}(N;\Z)=\Z$.
Using these last results in exact short sequence~\eqref{apAeq2}, one obtains before that $H_{2}(M;\Z)$ is a finitely generated Abelian group and after the equality
$$1=1+b_{2}(M)\,,$$
which implies $b_{2}(M)=0$. Then, the conclusion $H_{2}(M;\Z)=0$ follows, since $H_{2}(M;\Z)$ is torsion--free due to the fact that $M$ is noncompact.
Let us treat the second case, namely when $M$ is with boundary.
Again by virtue of the made assumption, $M$ is diffeomorphic to a $3$--dimensional, compact, smooth manifold $N$ having $\partial M$ as boundary, minus its point $P$. Moreover, one has the following exact short sequence
\begin{equation}\label{apAeq3}
0\to H_{3}(M,\partial M;\Z) \to H_{3}(N,\partial M;\Z)\to \Z\to H_{2}(M,\partial M;\Z)\to H_{2}(N,\partial M;\Z)\to 0\,,
\end{equation}
by relative Mayer--Vietoris. Therefore, if $H_{2}(M,\partial M;\Z)$ is zero, then $H_{2}(N,\partial M;\Z)$ is zero, which implies $b_{1}(N)=0$, by Lemma~\ref{compcase},
consequently, $b_{1}(M)=0$, since $H_{1}(M;\Z)$ and $H_{1}(N;\Z)$ are isomorphic groups.
Vice versa, due to this isomorphism, if $b_{1}(M)=0$, also $b_{1}(N)=0$, from which it follows that $N$ is orientable, $\partial M$ is a $2$--sphere and $H_{2}(N,\partial M;\Z)=0$, by Lemma~\ref{compcase}.
The orientability of $N$ ensures that $H_{3}(N,\partial M;\Z)=\Z$ and also the orientability of $M$, from which one obtains $H_{3}(M,\partial M;\Z)=0$, by the duality for noncompact manifolds.
Using all this information in exact short sequence~\eqref{apAeq3}, similarly to the case $\partial M=\emptyset$, one obtains $b_{2}(M,\partial M)=0$.
Then, the statement $H_{2}(M,\partial M;\Z)=0$ will follow once we show that $H_{2}(M,\partial M;\Z)$ is torsion--free.
We suppose, by contradiction, that there exists a nontrivial element of order $m$.
As a result, $\mathrm{Tor}\big(H_2(M,\partial M;\Z),\Z_{m}\big)$ is nonzero, but 
this is impossible since $H_3(M,\partial M; \Z_{m})$ and $\mathrm{Tor}\big(H_2(M,\partial M;\Z),\Z_{m}\big)$ are isomorphic, as a consequence of  the universal coefficient theorem for homology together with the equality $H_{3}(M,\partial M;\Z)=0$, and $H_3(M,\partial M; \Z_{m})=0$, thanks to the duality for noncompact manifolds.
\end{proof}

We point out that, if $M$ is a (smooth) manifold satisfying the assumptions of Proposition~\ref{noncompcase} and one of the previous equivalent conditions, then {\em $M$ is always orientable and $\partial M$ is a $2$--sphere if the boundary is present}.

We are now ready to present an alternative proof, with respect to the one presented in~\cite[at the end of Subsection~1.3]{AMMO}, of the connectedness of the regular level sets of $u\in C^{\infty}(M)$ solution of Dirichlet problem~\eqref{f1prelc4} in a $3$--dimensional, complete, one--ended asymptotically flat manifold with compact, connected boundary and having $H_2(M,\partial M;\Z)=0$. It is inspired by~\cite[Lemma~4.46]{DanLee} and~\cite[Lemma~2.3]{MuntWang}.

\begin{proposition}\label{conn.liv.reg}
Let $(M,g)$ be a $3$--dimensional, complete, one--ended asymptotically flat manifold with compact, connected boundary. 
Let $u\in C^{\infty}(M)$ be the solution of Dirichlet problem~\eqref{f1prelc4}. Assume that $H_2(M,\partial M;\Z)=0$, then, all regular level sets of $u$ are connected.
\end{proposition}
\begin{proof}
Let $t\in (0,1)$ be a regular value of $u$. It is obvious that $\{u\geq t\}=\overline{\{u>t\}}$ and $\{0\leq u\leq t\}=\overline{\{0<u<t\}}$, we want to see that they are connected.
First, we show the connectedness of $\{0 \leq u \leq t\} $. Supposing it is not connected, it must have a connected, compact component $K$ disjoint from $\partial M$. 
Then, $\partial K\subseteq \{u=t\}$ and, since $\{0\leq u\leq t\}=\overline{\{0<u<t\}}$, the interior of $K$ must be nonempty and contain some points where $0<u<t$, which is not possible, by the maximum principle.
On the other side, similarly, if $\{u \geq t\}$ is not connected, it must have a connected, compact component $K$, because there exists a compact set $\widetilde{K}$ of $M$ such that $M\setminus\widetilde{K}$ is contained in $\{u \geq t\}$, as a consequence of $u\to 1$ at $\infty$. Then, $\partial K\subseteq \{u=t\}$ and, since $\{u\geq t\}=\overline{\{u>t\}}$, the interior of $K$ must be nonempty and contain some points where $u>t$, which is not possible, by the maximum principle.
Hence, $\{u\geq t\}$ and $\{u\leq t\}$ are connected.
Let now $\varepsilon>0$ such that $[t-\varepsilon,t+\varepsilon]$ doesn't contain critical values of $u$, we consider the reduced Mayer--Vietoris exact sequence of the pair $\{0\leq u\leq t+\varepsilon\}$ and $\{u\geq t\}$,
\begin{equation*}
\dots\to H_{1}(M;\Z)\to\widetilde{H}_{0}\left(\{t\leq u\leq t+\varepsilon\};\Z\right)\to\widetilde{H}_{0}\left(\{0\leq u\leq t+\varepsilon\};\Z\right)\oplus\widetilde{H}_{0}\left(\{u\geq t\};\Z\right)\to \dotsi \,.
\end{equation*}
Then, as a consequence of the connectedness of the sets $\{0\leq u\leq t+\varepsilon\}$ and $\{u\geq t\}$, the last space, $\widetilde{H}_{0}\left(\{0\leq u\leq t+\varepsilon\};\Z\right)\oplus\widetilde{H}_{0}\left(\{u\geq t\};\Z\right)$, is trivial, therefore, $\widetilde{H}_{0}\left(\{t\leq u\leq t+\varepsilon\};\Z\right)$ is the image of ${H}_{1}(M;\Z)$, but this image is trivial. Indeed, the assumption $H_2(M,\partial M;\Z)=0$ implies that the first Betti number of $M$ is zero, by Proposition~\ref{noncompcase}, hence ${H}_{1}(M;\Z)$ coincides with its torsion subgroup, while $\widetilde{H}_{0}\left(\{t\leq u\leq t+\varepsilon\};\Z\right)$ is torsion--free (since ${H}_{0}(X;\Z)$ is isomorphic to $\widetilde{H}_{0}(X;\Z)\oplus\Z$ and ${H}_{0}(X;\Z)$ is isomorphic to a direct sum of $\Z$'s, one for each path--connected component of any topological space $X$).
Thus, $\widetilde{H}_{0}\left(\{t\leq u\leq t+\varepsilon\};\Z\right)=0$ and, consequently, $\{t\leq u\leq t+\varepsilon\}$ is connected, but, being $\{t\leq u\leq t+\varepsilon\}$ diffeomorphic to $\{u=t\}\times[t,t+\varepsilon]$, the number of the connected components of $\{t\leq u\leq t+\varepsilon\}$ and $\{u=t\}$ is the same.
\end{proof}

\bibliographystyle{amsplain}
\bibliography{biblio}
 
\end{document}